\documentclass[reqno]{amsart}
\usepackage[margin=1.15in]{geometry}
\usepackage{hyperref}
\usepackage{cite}
\usepackage{enumitem}
\usepackage{amssymb}

\theoremstyle{plain}
\newtheorem{theorem}{Theorem}
\newtheorem{proposition}[theorem]{Proposition}
\newtheorem{corollary}[theorem]{Corollary}
\newtheorem{lemma}[theorem]{Lemma}

\theoremstyle{definition}
\newtheorem{definition}{Definition}

\theoremstyle{remark}
\newtheorem{remark}{Remark}

\allowdisplaybreaks

\begin{document}

\author{Phuong Le}
\address{Phuong Le$^{1,2}$ (ORCID: 0000-0003-4724-7118)\newline
	$^1$Faculty of Economic Mathematics, University of Economics and Law, Ho Chi Minh City, Vietnam; \newline
	$^2$Vietnam National University, Ho Chi Minh City, Vietnam}
\email{phuongl@uel.edu.vn}


\title[Sharp Liouville thresholds and endpoint rigidity]{Sharp Liouville thresholds and endpoint rigidity for finite Morse index solutions of the $p$-Laplace Lane--Emden equation}

\begin{abstract}
We study $-\Delta_p u = |u|^{q-1}u$ in $\mathbb{R}^N$ with $N>p\ge 2$, for solutions that are stable outside a compact set, with no assumption of sign, boundedness or symmetry. Damascelli, Farina, Sciunzi and Valdinoci settled the subcritical range for $p>2$, and treated the supercritical range $p^*-1<q<q_c(N,p)$ only for radial solutions, stating that without radial symmetry they were not able to conclude. The obstruction is that the arguments available for $p=2$ rest on a monotonicity formula with no monotone analogue when $p\ne2$.

We remove the radial hypothesis: every $C^1$ weak solution stable outside a compact set is trivial when $p^*-1<q<q_c(N,p)$. The whole range $p-1<q<q_c(N,p)$ with $q\ne p^*-1$ is thereby settled, as it is for $p=2$ by Farina's theorem, and every finite Morse index solution in it is trivial. In its place we iterate the comparison principle against an explicit two-parameter supersolution, upgrading the decay forced by stability to that of the $p$-harmonic fundamental solution; this step uses no stability, only smallness of the self-similar decay, and is stated on its own. The upper threshold is sharp: for every $q>q_c(N,p)$ we exhibit positive bounded radial solutions stable outside a compact set. At the Sobolev endpoint, every $C^1$ solution stable outside a compact set has finite energy; for $p>2$ this removes the a priori $\mathcal D^{1,p}$ assumption from the low-Morse-index classification of Farina, Mercuri and Willem and answers their question on the Morse index of the Aubin--Talenti extremals: it is one. At the upper endpoint, $q_c(N,p)$ is the exact threshold for a nonzero homogeneous weak solution on $\mathbb R^N\setminus\{0\}$ to be stable outside a ball, and at the threshold there are exactly two.
\end{abstract}

\maketitle

\section{Introduction}\label{sec_intro}

Let $N>p>1$ and $q>p-1$. We consider the Lane--Emden equation driven by the $p$-Laplacian,
\begin{equation}\label{eq_main}
-\Delta_p u = |u|^{q-1}u \quad \text{in } \mathbb{R}^N,
\qquad \Delta_p u = \operatorname{div}\left(|\nabla u|^{p-2}\nabla u\right),
\end{equation}
and we study solutions that are stable outside a compact set. No sign condition, no boundedness assumption and no symmetry assumption are imposed. We write
\begin{equation}\label{eq_exponents}
p^*=\frac{Np}{N-p},\qquad p^*-1=\frac{N(p-1)+p}{N-p},\qquad
\alpha=\frac{p}{q-p+1},\qquad \gamma=\frac{N-p}{p-1}.
\end{equation}
The exponent $\alpha$ is the homogeneity of the scaling $u\mapsto \lambda^{\alpha}u(\lambda\,\cdot)$ that leaves \eqref{eq_main} invariant, and $\gamma$ is the decay rate of the fundamental solution of $-\Delta_p$. The Joseph--Lundgren exponent of the $p$-Laplacian is
\begin{equation}\label{eq_qc}
q_c(N,p)=
\begin{cases}
+\infty, & \text{if } (p-1)N\le p(p+3),\\[2mm]
\dfrac{
\begin{gathered}
\big((p-1)N-p\big)^2+p^2(p-2)-p^2(p-1)N\\[-1mm]
{}+2p^2\sqrt{(p-1)(N-1)}
\end{gathered}}
{(N-p)\big((p-1)N-p(p+3)\big)},
& \text{if } (p-1)N> p(p+3).
\end{cases}
\end{equation}
For $p=2$ this is the classical exponent of Joseph and Lundgren \cite{JosephLundgren1973}. For general $p>1$ a threshold playing the same role appeared earlier in the work of Cabr\'e and Sanch\'on on semi-stable singular power solutions of the Dirichlet problem on a ball \cite[Prop.~1.3]{CabreSanchon2007}, where it is denoted by $m_{cs}(p)$. One always has $q_c(N,p)>p^*-1$.

For $p=2$ the picture is complete and is due to Farina \cite{Farina2007}. Every $C^2$ solution of $-\Delta u=|u|^{q-1}u$ in $\mathbb{R}^N$ that is stable outside a compact set is trivial as soon as $q<q_c(N,2)$ and $q\ne \frac{N+2}{N-2}$, and both exclusions are necessary. At $q=\frac{N+2}{N-2}$ the Aubin--Talenti bubbles have Morse index one, and for $q\ge q_c(N,2)$ with $N\ge 11$ there are positive radial stable solutions. Earlier, Bahri and Lions \cite{BahriLions1992} had obtained the subcritical part of this statement for finite Morse index solutions.

For $p\ne 2$ the situation is different, and the difference is confined to the supercritical range. The reference point for the entire-space Liouville theory is the work of Damascelli, Farina, Sciunzi and Valdinoci \cite{DFSV2009}, where the same threshold $q_c(N,p)$ appears as the upper endpoint of their nonexistence range for stable solutions. Assuming $p>2$ throughout, they proved that every stable solution of \eqref{eq_main} is trivial when $p-1<q<q_c(N,p)$ \cite[Thm.~1.5]{DFSV2009}, and that every solution which is stable outside a compact set is trivial when $p-1<q<p^*-1$ \cite[Thm.~1.9]{DFSV2009}. In the supercritical range
\begin{equation}\label{eq_superrange}
p^*-1<q<q_c(N,p)
\end{equation}
they prove two things. First, every solution stable outside a compact set decays faster than the self-similar rate, namely $|x|^{p/(q-p+1)}|u(x)|\to0$ \cite[Thm.~1.10]{DFSV2009}. Second, using that decay, every \emph{radial} solution stable outside a compact set is trivial \cite[Thm.~1.11]{DFSV2009}. For a solution that is not assumed radial they state explicitly that they are not able to conclude. To the best of our knowledge, no subsequent result has removed this radial restriction in the full supercritical range \eqref{eq_superrange} for \eqref{eq_main}. Analogous restrictions also occur in related anisotropic extensions; see, for instance, \cite{FSV2024,LWY2025}.

The purpose of this paper is to remove that restriction.

Stability is also the driving hypothesis of a second and very active line of work, in which the question is regularity rather than nonexistence. On a bounded domain, Cabr\'e, Figalli, Ros-Oton and Serra \cite{CFRS2020} proved that stable solutions of $-\Delta u=f(u)$ with $f$ nonnegative are bounded, and therefore smooth, in dimension $N\le9$, by an interior estimate that depends only on the $L^1$ norm of the solution and not on the nonlinearity, and the dimension $9$ is optimal. For the $p$-Laplacian with $p\ge2$ the corresponding interior estimate is due to Cabr\'e, Miraglio and Sanch\'on \cite{CMS2022}, after Castorina and Sanch\'on \cite{CastorinaSanchon2015}. Their restriction on the dimension is $N<p+\frac{4p}{p-1}$, that is $(p-1)N<p(p+3)$, and this is precisely the condition under which $q_c(N,p)=+\infty$ in \eqref{eq_qc}. The two lines of work therefore stop at the same threshold, which appears once as a restriction on the dimension for an arbitrary nonlinearity and once as a restriction on the exponent for $f(u)=|u|^{q-1}u$. For finite Morse index solutions of the Lane--Emden equation itself, the regularity side was studied by D\'avila, Dupaigne and Farina \cite{DDF2011}.

\smallskip
Stability outside a compact set $K$ yields, by a test function argument going back to \cite{Farina2007} and carried out for the $p$-Laplacian in \cite[Prop.~1.4]{DFSV2009}, the family of estimates
\begin{equation}\label{eq_intro_estimate}
\int_{B_{2r}\setminus B_r}|u|^{q+\delta}\,dx \le C\, r^{\,N-\frac{p(q+\delta)}{q-p+1}},
\qquad \delta\in[1,\delta_0),\qquad
\delta_0=\frac{2q-(p-1)+2\sqrt{q(q-p+1)}}{p-1}.
\end{equation}
To reach a Liouville theorem one wants the finite energy conditions $|\nabla u|^p\in L^1(\mathbb{R}^N)$ and $|u|^{q+1}\in L^1(\mathbb{R}^N)$, after which the Pohozaev identity forces $u\equiv 0$ unless $q=p^*-1$. The choice $\delta=1$ in \eqref{eq_intro_estimate} is the only one that produces exactly those two quantities, and the corresponding exponent is
\[
N-\frac{p(q+1)}{q-p+1},
\]
which is negative precisely when $q<p^*-1$. This is the reason, and the only reason, why the results quoted above stop at the Sobolev exponent once the radial hypothesis is dropped.

For $p=2$ the supercritical range is reached by a different route. One first deduces from \eqref{eq_intro_estimate} that $|u(x)|=o(|x|^{-\alpha})$, and then one uses the monotonicity formula of Pacard type: the rescaled energy
\[
E(r)=r^{\frac{2(q+1)}{q-1}-N}\left(\frac12\int_{B_r}|\nabla u|^2-\frac{1}{q+1}\int_{B_r}|u|^{q+1}\right)
+\frac{1}{q-1}\,r^{\frac{4}{q-1}+1-N}\int_{\partial B_r}u^2\,d\sigma
\]
is nondecreasing, it tends to $0$ at both ends, hence it vanishes identically and $u$ is homogeneous of a negative degree, so $u\equiv 0$. This is the mechanism used by Dancer, Du and Guo \cite{DDG2011}, by D\'avila, Dupaigne, Wang and Wei \cite{DDWW2014} in the fourth order case, and by Chen and Liao \cite{ChenLiao2025} for the Baouendi--Grushin operator.

That mechanism has no counterpart for the $p$-Laplacian. Write $u_r(x)=r^{\alpha}u(rx)$, which solves \eqref{eq_main} again, and set
\[
E(r)=r^{\frac{p(q+1)}{q-p+1}-N}\left(\frac1p\int_{B_r}|\nabla u|^p-\frac{1}{q+1}\int_{B_r}|u|^{q+1}\right)
=\frac1p\int_{B_1}|\nabla u_r|^p-\frac{1}{q+1}\int_{B_1}|u_r|^{q+1}.
\]
Differentiating in $r$ and integrating by parts on $B_1$ gives
\begin{equation}
\frac{dE}{dr}=\int_{\partial B_1}|\nabla u_r|^{p-2}\left[\,r\,\dot u_r^{\,2}-\alpha\,u_r\,\dot u_r\,\right]d\sigma ,
\qquad \dot u_r=\partial_r u_r .
\end{equation}
When $p=2$ the second term inside the bracket is $-\frac{\alpha}{2}\frac{d}{dr}u_r^2$, and moving it to the left produces the boundary correction that makes the formula monotone. When $p\ne 2$ the weight $|\nabla u_r|^{p-2}$ itself depends on $r$, the product $|\nabla u_r|^{p-2}u_r\dot u_r$ is not an exact derivative in $r$, and no correction term of the same type is available. This prevents the standard boundary correction from producing a monotonicity formula and is one of the main obstacles in treating the supercritical range for $p\ne2$; to the best of our knowledge, the corresponding nonradial problem has remained open since \cite{DFSV2009}.

Our observation is that the monotonicity formula is not what the argument really needs. What it needs is finite energy, and finite energy can be obtained from a decay rate, provided the rate is the one of the fundamental solution of $-\Delta_p$ rather than the self-similar one. Since $q>p^*-1$ implies $q>\frac{N(p-1)}{N-p}$, we have
\[
\alpha<\gamma,
\]
and the gap between the two rates is exactly the room in which a bootstrap can operate. The bootstrap is driven by the weak comparison principle applied to the explicit two-parameter family
\[
w(x)=A|x|^{-\alpha}+B|x|^{-\gamma},
\]
for which $-\Delta_p w$ can be computed in closed form; see Lemma \ref{lem_super}. The second term is $p$-harmonic and contributes nothing to $-\Delta_p w$ beyond a factor larger than $1$ when $p\ge 2$, so it can be used freely to carry the boundary datum while the first term absorbs the nonlinearity. Iterating produces a sequence of radii and amplitudes whose bookkeeping closes exactly at the rate $|x|^{-\gamma}$, up to a power of $\log|x|$ which is harmless.

\smallskip
We work with weak solutions in the following sense. A function $u\in W^{1,p}_{\mathrm{loc}}(\mathbb{R}^N)\cap L^q_{\mathrm{loc}}(\mathbb{R}^N)$ is a weak solution of \eqref{eq_main} if
\[
\int_{\mathbb{R}^N}|\nabla u|^{p-2}\nabla u\cdot\nabla\varphi\,dx=\int_{\mathbb{R}^N}|u|^{q-1}u\,\varphi\,dx
\qquad\text{for every }\varphi\in C_c^1(\mathbb{R}^N).
\]
By the regularity theory of DiBenedetto \cite{DiBenedetto1983} and Tolksdorf \cite{Tolksdorf1984}, a locally bounded weak solution belongs to $C^{1,\beta}_{\mathrm{loc}}(\mathbb{R}^N)$ for some $\beta\in(0,1)$, so the hypothesis $u\in C^1(\mathbb{R}^N)$ made in Theorem \ref{thm_main} below is the same as local boundedness. Theorem \ref{thm_sing} concerns weak solutions that are not locally bounded, which is why the definition is not restricted to that class. The linearized quadratic form at $u$ is
\[
Q_u(\varphi)=\int_{\mathbb{R}^N}|\nabla u|^{p-2}|\nabla\varphi|^2\,dx
+(p-2)\int_{\mathbb{R}^N}|\nabla u|^{p-4}(\nabla u\cdot\nabla\varphi)^2\,dx
-q\int_{\mathbb{R}^N}|u|^{q-1}\varphi^2\,dx,
\]
where, when $p>2$, the integrands of the first two terms are understood to vanish on $\{\nabla u=0\}$, which is their continuous extension because both are bounded by $|\nabla u|^{p-2}|\nabla\varphi|^2$; for $p=2$ the second term vanishes identically and $Q_u$ is the classical form $\int|\nabla\varphi|^2-q\int|u|^{q-1}\varphi^2$. All three integrals are finite for every $\varphi\in C_c^1(\mathbb{R}^N)$: the first two because $|\nabla u|^{p-2}\in L^{\frac{p}{p-2}}_{\mathrm{loc}}(\mathbb{R}^N)$ when $p>2$, the case $p=2$ being immediate, and the third because $|u|^{q-1}\in L^{\frac{q}{q-1}}_{\mathrm{loc}}(\mathbb{R}^N)$.

\begin{definition}\label{def_stable}
Let $u$ be a weak solution of \eqref{eq_main} and let $K\subset\mathbb{R}^N$ be compact. We say that $u$ is \emph{stable outside $K$} if $Q_u(\varphi)\ge 0$ for every $\varphi\in C_c^1(\mathbb{R}^N\setminus K)$. We say that $u$ has \emph{finite Morse index} if there is an integer $k$ such that no subspace of $C_c^1(\mathbb{R}^N)$ of dimension larger than $k$ consists of functions $\varphi$ with $Q_u(\varphi)<0$ apart from $\varphi=0$, and we write $i(u)$ for the least such $k$, the \emph{Morse index} of $u$.
\end{definition}

\begin{remark}\label{rem_fmi}
Every finite Morse index solution is stable outside a compact set. Indeed, let $u$ have finite Morse index. The dimensions of the subspaces $Z\subset C_c^1(\mathbb{R}^N)$ such that $Q_u(\varphi)<0$ for every $\varphi\in Z\setminus\{0\}$ are bounded, so one of them, say $W$, has maximal dimension, the zero subspace being admissible. Since $W$ is finite dimensional and each of its elements has compact support, there is $R>0$ such that every element of $W$ is supported in $B_R$. Suppose $Q_u(\varphi)<0$ for some $\varphi\in C_c^1(\mathbb{R}^N\setminus\overline{B_R})$. Then $w$ and $\varphi$ have disjoint supports for every $w\in W$, so $Q_u(w+t\varphi)=Q_u(w)+t^2Q_u(\varphi)<0$ whenever $(w,t)\in W\times\mathbb{R}$ and $(w,t)\ne(0,0)$. Thus $W\oplus\mathbb{R}\varphi$ contradicts the maximality of $W$. Hence $Q_u\ge0$ on $C_c^1(\mathbb{R}^N\setminus\overline{B_R})$.
\end{remark}

\begin{theorem}\label{thm_main}
Let $N>p\ge 2$ and let $q$ satisfy
\[
p^*-1<q<q_c(N,p).
\]
Then every weak solution $u\in C^1(\mathbb{R}^N)$ of \eqref{eq_main} that is stable outside a compact set satisfies $u\equiv 0$.
\end{theorem}

Stability enters the proof of Theorem \ref{thm_main} only to produce a decay rate. Once that rate is available, the rest of the argument is independent of it, and it applies on a wider range of exponents.

\begin{theorem}\label{thm_decay}
Let $N>p\ge2$ and $q>\frac{N(p-1)}{N-p}$. There is $\varepsilon_*=\varepsilon_*(N,p,q)>0$ with the following property. Let $u\in C^1(\mathbb R^N)$ be a weak solution of \eqref{eq_main} such that
\[
\limsup_{|x|\to\infty}|x|^{\alpha}|u(x)|<\varepsilon_* .
\]
Then
\[
u\in\mathcal D^{1,p}(\mathbb R^N),\qquad
\int_{\mathbb R^N}|\nabla u|^p\,dx=\int_{\mathbb R^N}|u|^{q+1}\,dx<\infty,
\]
and if moreover $q\ne p^*-1$, then $u\equiv0$.
\end{theorem}

Here $\mathcal D^{1,p}(\mathbb R^N)=\{v\in L^{p^*}(\mathbb R^N):\nabla v\in L^p(\mathbb R^N)\}$. No stability is assumed in Theorem \ref{thm_decay}, and its range of exponents is strictly larger than that of Theorem \ref{thm_main}, since $\frac{N(p-1)}{N-p}<p^*-1$: it reaches below the Sobolev exponent, where a decay hypothesis replaces stability, and it has no upper restriction, the Joseph--Lundgren exponent entering only through Proposition \ref{prop_decay}. The smallness of the self-similar decay, rather than its vanishing, is all that the comparison bootstrap uses. Theorem \ref{thm_main} follows by feeding Theorem \ref{thm_decay} the decay that stability provides.

Combining Theorem \ref{thm_main} with the subcritical theorem of \cite{DFSV2009} gives the statement that matches Farina's theorem for $p=2$.

\begin{corollary}\label{cor_full}
Let $N>p\ge 2$ and let $q$ satisfy $p-1<q<q_c(N,p)$ and $q\ne p^*-1$. Then every weak solution $u\in C^1(\mathbb{R}^N)$ of \eqref{eq_main} that is stable outside a compact set satisfies $u\equiv 0$. In particular, every finite Morse index solution of \eqref{eq_main} in that range of exponents is trivial.
\end{corollary}

The Sobolev exclusion in Corollary \ref{cor_full} is necessary. At $q=p^*-1$ the functions
\begin{equation}\label{eq_bubble}
u_{\lambda,x_0}(x)=\pm\left(\frac{\lambda^{\frac{1}{p-1}}N^{\frac1p}\left(\frac{N-p}{p-1}\right)^{\frac{p-1}{p}}}
{\lambda^{\frac{p}{p-1}}+|x-x_0|^{\frac{p}{p-1}}}\right)^{\frac{N-p}{p}},
\qquad \lambda>0,\ x_0\in\mathbb{R}^N,
\end{equation}
are nontrivial solutions of finite Morse index, and they are stable outside a ball by \cite[Prop.~10.1]{DFSV2009}; their Morse index is computed in Theorem \ref{thm_index} below. The upper threshold $q_c(N,p)$ has a different sharpness mechanism in the larger weak class allowing an isolated singularity. For $q>p^*-1$ put
\begin{equation}\label{eq_cs}
c_s=\Big[\alpha^{p-1}(p-1)(\gamma-\alpha)\Big]^{\frac{1}{q-p+1}},
\qquad
u_{x_0}(x)=c_s\,|x-x_0|^{-\alpha},\quad x_0\in\mathbb{R}^N .
\end{equation}

\begin{theorem}\label{thm_sing}
Let $N>p\ge2$ and $q>p^*-1$. Then every $u_{x_0}$ is a positive weak solution of \eqref{eq_main} belonging to $W^{1,p}_{\mathrm{loc}}(\mathbb{R}^N)\cap C^\infty(\mathbb{R}^N\setminus\{x_0\})$, and the following are equivalent:
\begin{enumerate}[label=(\roman*)]
\item $Q_{u_{x_0}}(\varphi)\ge0$ for every $\varphi\in C_c^1(\mathbb{R}^N\setminus\{x_0\})$;
\item $u_{x_0}$ is stable outside some compact set;
\item $q\ge q_c(N,p)$.
\end{enumerate}
\end{theorem}

Theorem \ref{thm_sing} exhibits, for every $q\ge q_c(N,p)$, an $N$-parameter family of nontrivial positive weak solutions of \eqref{eq_main} whose quadratic form is nonnegative away from the singular point. These profiles are not in the $C^1(\mathbb R^N)$ class of Theorem \ref{thm_main}; regular counterexamples are supplied instead by Corollary \ref{cor_radialsharp} below, for every $q>q_c(N,p)$. For $p=2$ the equivalence of (i) and (iii) is the classical Joseph--Lundgren stability criterion. For general $p>1$, Cabr\'e and Sanch\'on had already met the same threshold in their study of the singular function $U^{\#}(x)=|x|^{-p/(m-p+1)}-1$ on the unit ball, which they prove to be a semi-stable solution of the associated Dirichlet problem exactly when $m\ge m_{cs}(p)$ \cite[Prop.~1.3]{CabreSanchon2007}. Theorem \ref{thm_sing} is recorded here in the normalization and whole-space framework of \eqref{eq_main}, and its direct weighted-Hardy proof will also be useful for Proposition \ref{prop_radial}. We make no claim of a regular $C^1$ counterexample at the exact endpoint $q=q_c(N,p)$ when $p>2$.

Theorem \ref{thm_sing} concerns one profile. In fact $q_c(N,p)$ is the exact threshold for the existence of \emph{any} stable homogeneous profile, with no sign assumption, and at the threshold the profiles of \eqref{eq_cs} are the only ones. Call $U$ \emph{$(-\alpha)$-homogeneous} if $U(\lambda x)=\lambda^{-\alpha}U(x)$ for every $\lambda>0$.

\begin{theorem}[Homogeneous profiles and the exponent $q_c$]\label{thm_qc_hom}
Let $N>p\ge2$ and $q>p^*-1$. Let $U\in W^{1,p}_{\mathrm{loc}}(\mathbb R^N\setminus\{0\})\cap L^{q+1}_{\mathrm{loc}}(\mathbb R^N\setminus\{0\})$ be a $(-\alpha)$-homogeneous weak solution of \eqref{eq_main} in $\mathbb R^N\setminus\{0\}$ such that $Q_U(\varphi)\ge0$ for every $\varphi\in C_c^1(\mathbb R^N\setminus\overline{B_R})$ and some $R>0$. Then
\begin{enumerate}[label=\rm(\roman*)]
\item if $q<q_c(N,p)$, then $U\equiv0$;
\item if $q=q_c(N,p)$, then $U\equiv0$ or $U(x)=\pm c_s|x|^{-\alpha}$, with $c_s$ given by \eqref{eq_cs}.
\end{enumerate}
\end{theorem}

Part (i) covers the case $q_c(N,p)=+\infty$, in which it says that no nonzero stable homogeneous profile exists for any $q>p^*-1$. Taken together with Theorem \ref{thm_sing}, the two parts say that a nonzero $(-\alpha)$-homogeneous weak solution whose quadratic form is nonnegative outside a ball exists if and only if $q\ge q_c(N,p)$, and that at $q=q_c(N,p)$ there are exactly two. The profiles $\pm c_s|x|^{-\alpha}$ are therefore not merely an example at the threshold: they are the whole of it. Part (i) is also what proves the implication (ii) $\Rightarrow$ (iii) of Theorem \ref{thm_sing}, since $u_s$ is $(-\alpha)$-homogeneous and nonzero. There is no circularity: the proof of Theorem \ref{thm_qc_hom} uses neither Theorem \ref{thm_sing} nor Lemma \ref{lem_hardy}. For $q>q_c(N,p)$ the argument stops: the quantity that vanishes at the threshold has the wrong sign, and we do not know whether other homogeneous stable profiles appear.

In Section \ref{sec_sing} we also prove compactness of arbitrary blow-down sequences at $q=q_c$. What is not supplied by the present argument is a monotonicity formula forcing every blow-down limit to be homogeneous.

The solutions $u_{x_0}$ are unbounded near $x_0$, so they do not contradict Theorem \ref{thm_main}, whose hypothesis $u\in C^1(\mathbb{R}^N)$ they fail. Regular solutions with the same behavior at infinity exist as soon as $q>q_c(N,p)$, and they too are stable outside a compact set. We say that a positive radial solution $u$ of \eqref{eq_main} on $\mathbb R^N$ has \emph{slow decay} if
\begin{equation}\label{eq_slow}
0<\liminf_{r\to\infty}r^{\alpha}u(r)\le\limsup_{r\to\infty}r^{\alpha}u(r)<\infty .
\end{equation}
This is a statement about the rate only; the constant is not prescribed. It turns out to be prescribed automatically, together with the gradient.

\begin{proposition}\label{prop_radial}
Let $N>p\ge2$ and $q>p^*-1$, and let $u\in C^1(\mathbb{R}^N)$ be a positive radial solution of \eqref{eq_main} with slow decay \eqref{eq_slow}. Then
\begin{equation}\label{eq_slowexact}
\lim_{|x|\to\infty}|x|^{\alpha}u(x)=c_s
\qquad\text{and}\qquad
\lim_{|x|\to\infty}|x|^{\alpha+1}|\nabla u(x)|=\alpha c_s .
\end{equation}
If moreover $q>q_c(N,p)$, then $u$ is stable outside a compact set.
\end{proposition}

The proof of \eqref{eq_slowexact} is a two-dimensional phase-plane argument whose whole input is that the divergence of the relevant vector field is the constant $p\alpha-(N-p)$, negative precisely because $q>p^*-1$. The only fact we import is the existence of radial solutions with slow decay, which is classical.

\begin{corollary}\label{cor_radialsharp}
Let $N>p\ge2$ and $q>q_c(N,p)$. Then for every $a>0$ equation \eqref{eq_main} has a positive bounded radial solution $u_a\in C^1(\mathbb{R}^N)$ with $u_a(0)=a$, and every such solution satisfies \eqref{eq_slowexact} and is stable outside a compact set. In particular the upper endpoint in Theorem \ref{thm_main} cannot be moved past $q_c(N,p)$, even within the class of positive bounded radial solutions.
\end{corollary}

\begin{proof}
For $q>p^*-1$ and every $a>0$ there is a positive radial solution of \eqref{eq_main} on $\mathbb R^N$ with $u_a(0)=a$ and $u_a'(0)=0$, and every such solution has slow decay \eqref{eq_slow}; this is the classification of radial solutions of \eqref{eq_main} due to Bidaut-V\'eron \cite{BidautVeron1989}, in the form given by Franca \cite[Cor.~3.14]{Franca2004}. Since $q>q_c(N,p)>p^*-1$, Proposition \ref{prop_radial} applies. Such a solution is radially decreasing, because the right hand side of \eqref{eq_radialint} below is positive, so it is bounded by $a$.
\end{proof}

The other excluded exponent has a different structure. At $q=p^*-1$ the Pohozaev identity no longer forces triviality. Nevertheless the decay argument developed for Theorem \ref{thm_main} still yields finite energy, without imposing it a priori.

\begin{theorem}[The Sobolev endpoint]\label{thm_sob_endpoint}
Let $N>p\ge2$, $q=p^*-1$, and let $u\in C^1(\mathbb R^N)$ be a weak solution of \eqref{eq_main} which is stable outside a compact set. Then
\begin{equation}\label{eq_critical_energy}
u\in\mathcal D^{1,p}(\mathbb R^N),\qquad
\int_{\mathbb R^N}|\nabla u|^p\,dx
=\int_{\mathbb R^N}|u|^{p^*}\,dx<\infty.
\end{equation}
If $u$ has finite Morse index, then the number of its nodal regions is finite and does not exceed $i(u)$. In particular, if $i(u)\le1$, then either $u\equiv0$ or, up to sign, $u=u_{\lambda,x_0}$ for some $\lambda>0$ and $x_0\in\mathbb R^N$; and the same alternative holds for a solution of one sign, without assuming finite Morse index.
\end{theorem}

For $p=2$ the first assertion of Theorem \ref{thm_sob_endpoint} is due to Farina \cite[Thm.~2]{Farina2007}, who proves it for $C^2$ solutions by a different route; what is new here is the range $p>2$, where no monotonicity formula is available. For $p>2$, the estimate on the number of nodal regions is \cite[Thm.~1.4]{FMW2019}; its hypothesis $u\in\mathcal D^{1,p}(\mathbb R^N)$ is precisely what \eqref{eq_critical_energy} supplies and what is assumed a priori in \cite{FMW2019}. For $p=2$ the same estimate follows directly from the standard nodal-domain localization argument; we include it in the proof below. Once a solution has at most one nodal region it does not change sign, and the classification of positive finite-energy solutions of the critical equation gives an Aubin--Talenti profile: this is \cite{Sciunzi2016} for $p>2$ and the theorem of Caffarelli, Gidas and Spruck \cite{CGS1989} for $p=2$. For higher index a formula classification is not expected from energy information alone: nonradial sign-changing finite-energy solutions at the critical exponent are known to exist for $N\ge4$ \cite{ClappRios2018}. In fact, when $N\ge4$, Clapp and Vicente-Ben\'itez recently proved that for every $1<p<N$ the critical $p$-Laplace equation in $\mathbb R^N$ possesses infinitely many nonradial sign-changing solutions which are inequivalent under translations and dilations \cite[Thm.~1.1]{ClappVicente2026}.

The functions $u_{\lambda,x_0}$ themselves are solutions of finite Morse index, and they are stable outside a ball by \cite[Prop.~10.1]{DFSV2009}. Farina, Mercuri and Willem observed that the exact value of their Morse index was not known \cite[Rem.~1]{FMW2019}. It is one.

\begin{theorem}\label{thm_index}
Let $2<p<N$ and $q=p^*-1$. Then $i(u_{\lambda,x_0})=1$ for every $\lambda>0$ and every $x_0\in\mathbb{R}^N$.
\end{theorem}

Only the upper bound has to be proved: the lower bound $i(u_{\lambda,x_0})\ge1$ is contained in \cite[Thm.~1.4]{FMW2019}, which bounds the number of nodal regions of a solution by its Morse index. The upper bound is a constrained second-order argument on the Nehari manifold, and the role of the hypothesis $p>2$ is that it makes the energy functional twice differentiable.

The proof of Theorem \ref{thm_main} does not use the hypothesis $p\ge 2$ in an essential way at every step, but it does use it in a way we have not been able to remove, namely in the integral estimate of Lemma \ref{lem_stab} and in the sign of the exponent $p-2$ in Lemma \ref{lem_super}. See Remark \ref{rem_psmall}.

\begin{remark}
For $p=2$ the argument below is a proof of the supercritical part of Farina's theorem \cite{Farina2007} which uses no monotonicity formula and no blow-down analysis. Sections \ref{sec_decay} to \ref{sec_pohozaev} may be read with $p=2$ throughout, in which case Lemma \ref{lem_super} reduces to the elementary computation $-\Delta\big(A|x|^{-\alpha}+B|x|^{2-N}\big)=\alpha(N-2-\alpha)A|x|^{-\alpha-2}$.
\end{remark}

\smallskip
Section \ref{sec_prelim} collects the comparison principle, the subsolution property of $|u|$ and the local estimates that we use. Section \ref{sec_stab} proves the integral estimates coming from stability and identifies the exponent range in which they are useful, which is where $q_c(N,p)$ enters. Section \ref{sec_decay} proves the decay $|u(x)|=o(|x|^{-\alpha})$ together with its gradient counterpart. Section \ref{sec_boot} contains the comparison bootstrap, which is the new ingredient. Section \ref{sec_pohozaev} deduces finite energy, applies the Pohozaev identity and proves Theorems \ref{thm_decay} and \ref{thm_main}. Section \ref{sec_sing} proves Theorems \ref{thm_sing} and \ref{thm_qc_hom}, establishes blow-down compactness at $q=q_c$, and proves Proposition \ref{prop_radial}. Section \ref{sec_index} proves Theorem \ref{thm_sob_endpoint} and computes the Morse index of the extremals in Theorem \ref{thm_index}.

Throughout, $B_r$ denotes the open Euclidean ball of radius $r$ centered at the origin, and $C$ denotes a positive constant that may change from line to line and depends only on $N$, $p$, $q$ and on the quantities indicated.

\section{Preliminaries}\label{sec_prelim}

\begin{lemma}[Weak comparison principle on exterior domains]\label{lem_comp}
Let $R>0$, let $f\in L^\infty_{\mathrm{loc}}(\overline{B_R^c})$ where $B_R^c=\{x:|x|>R\}$, and let
$v,w\in C(\overline{B_R^c})\cap W^{1,p}_{\mathrm{loc}}(B_R^c)$ satisfy
\[
\int_{B_R^c}|\nabla v|^{p-2}\nabla v\cdot\nabla\varphi\,dx\le\int_{B_R^c}f\varphi\,dx
\le\int_{B_R^c}|\nabla w|^{p-2}\nabla w\cdot\nabla\varphi\,dx
\]
for every nonnegative $\varphi\in C_c^1(B_R^c)$. Assume $v\le w$ on $\partial B_R$ and
$\limsup_{|x|\to\infty}\big(v(x)-w(x)\big)\le 0$. Then $v\le w$ in $B_R^c$.
\end{lemma}

\begin{proof}
Fix $\lambda>0$ and set $D_\lambda=\{x\in B_R^c: v(x)-w(x)>\lambda\}$. Since $v-w$ is continuous on $\overline{B_R^c}$ and $v-w\le 0$ on $\partial B_R$, the set $D_\lambda$ has positive distance from $\partial B_R$. Since $\limsup_{|x|\to\infty}(v-w)\le 0$, the set $D_\lambda$ is bounded. Hence $(v-w-\lambda)^+$ belongs to $W^{1,p}(B_R^c)$ and has compact support in $B_R^c$, so it is an admissible test function by density. Subtracting the two inequalities and testing with it gives
\[
\int_{D_\lambda}\left(|\nabla v|^{p-2}\nabla v-|\nabla w|^{p-2}\nabla w\right)\cdot(\nabla v-\nabla w)\,dx\le 0 .
\]
The map $\xi\mapsto|\xi|^{p-2}\xi$ is strictly monotone on $\mathbb{R}^N$, so the integrand is nonnegative and vanishes only where $\nabla v=\nabla w$. Therefore $\nabla(v-w)=0$ almost everywhere in $D_\lambda$, so $v-w$ is constant on each connected component of $D_\lambda$. On the boundary of any such component one has $v-w=\lambda$ by continuity, while $v-w>\lambda$ inside. Hence $D_\lambda=\emptyset$. Letting $\lambda\to 0$ finishes the proof.
\end{proof}

\begin{lemma}\label{lem_abs}
Let $u\in C^1(\mathbb{R}^N)$ be a weak solution of \eqref{eq_main}. Then $|u|$ is a weak subsolution of $-\Delta_p z=|u|^q$, that is,
\[
\int_{\mathbb{R}^N}\big|\nabla|u|\big|^{p-2}\nabla|u|\cdot\nabla\varphi\,dx\le\int_{\mathbb{R}^N}|u|^q\varphi\,dx
\qquad\text{for every nonnegative }\varphi\in C_c^1(\mathbb{R}^N).
\]
\end{lemma}

\begin{proof}
For $\varepsilon>0$ put $\theta_\varepsilon(s)=s(s^2+\varepsilon^2)^{-1/2}$, so that $\theta_\varepsilon$ is smooth, bounded by $1$ in absolute value, and $\theta_\varepsilon'\ge 0$. Let $\varphi\in C_c^1(\mathbb{R}^N)$ be nonnegative. The function $\theta_\varepsilon(u)\varphi$ is Lipschitz with compact support, hence admissible, and testing \eqref{eq_main} with it gives
\[
\int_{\mathbb{R}^N}|\nabla u|^{p-2}\theta_\varepsilon(u)\,\nabla u\cdot\nabla\varphi\,dx
+\int_{\mathbb{R}^N}|\nabla u|^{p}\,\theta_\varepsilon'(u)\,\varphi\,dx
=\int_{\mathbb{R}^N}|u|^{q-1}u\,\theta_\varepsilon(u)\,\varphi\,dx .
\]
The second term on the left is nonnegative, and $|u|^{q-1}u\,\theta_\varepsilon(u)\le|u|^q$ pointwise. Hence
\[
\int_{\mathbb{R}^N}|\nabla u|^{p-2}\theta_\varepsilon(u)\,\nabla u\cdot\nabla\varphi\,dx\le\int_{\mathbb{R}^N}|u|^q\varphi\,dx .
\]
As $\varepsilon\to 0$ we have $\theta_\varepsilon(u)\to \operatorname{sgn}(u)$ pointwise, with the convention $\operatorname{sgn}(0)=0$. Since $\nabla u=0$ almost everywhere on $\{u=0\}$ and $\nabla|u|=\operatorname{sgn}(u)\nabla u$ almost everywhere, the integrand converges almost everywhere to $\big|\nabla|u|\big|^{p-2}\nabla|u|\cdot\nabla\varphi$ and is dominated by $|\nabla u|^{p-1}|\nabla\varphi|$, which is integrable on the support of $\varphi$. Dominated convergence gives the claim.
\end{proof}

We shall use two local estimates for the $p$-Laplacian. The first is Serrin's local boundedness theorem.

\begin{lemma}[{Serrin \cite[Thm.~1]{Serrin1964}, see also \cite{Trudinger1967}}]\label{lem_serrin}
Let $\Omega=B_4\setminus \overline{B_1}$ and $\Omega'=B_3\setminus\overline{B_2}$, let $p>1$ and let $k>N/p$. For every $M>0$ there is a constant $C=C(N,p,k,M)$ with the following property. If $c\in L^k(\Omega)$ with $\|c\|_{L^k(\Omega)}\le M$ and if $v\in W^{1,p}(\Omega)\cap C(\Omega)$ satisfies
\[
\int_{\Omega}|\nabla v|^{p-2}\nabla v\cdot\nabla\varphi\,dx=\int_{\Omega}c(x)|v|^{p-2}v\,\varphi\,dx
\qquad\text{for every }\varphi\in C_c^1(\Omega),
\]
then
\[
\sup_{\Omega'}|v|\le C\,\|v\|_{L^{p}(\Omega)} .
\]
\end{lemma}

The equation above is of the form $\operatorname{div}\mathbf A(x,v,\nabla v)=B(x,v,\nabla v)$ treated in \cite{Serrin1964}, with $\mathbf A(x,z,\xi)=|\xi|^{p-2}\xi$ and $B(x,z,\xi)=-c(x)|z|^{p-2}z$. The only coefficient occurring in the structure conditions of \cite{Serrin1964} that is not zero here is the one multiplying $|z|^{p-1}$ in the bound for $|B|$, and it equals $|c|$. Its scaling critical space is $L^{N/p}$, so the hypothesis $k>N/p$ puts it strictly above the critical exponent, which is what the local boundedness estimate requires. The dependence of the constant on $c$ enters only through $\|c\|_{L^k}$, whence the uniformity recorded above.

The second is the interior gradient bound, in the scale invariant form that we need.

\begin{lemma}\label{lem_grad}
Let $p>1$. There is $C=C(N,p)$ such that for every $v\in C^1(B_2)$ solving $-\Delta_p v=g$ weakly in $B_2$ with $g\in L^\infty(B_2)$ one has
\[
\sup_{B_1}|\nabla v|\le C\left(\|v\|_{L^\infty(B_2)}+\|g\|_{L^\infty(B_2)}^{\frac{1}{p-1}}\right).
\]
\end{lemma}

\begin{proof}
Set $\lambda=\|v\|_{L^\infty(B_2)}+\|g\|_{L^\infty(B_2)}^{1/(p-1)}$, which we may assume positive, and put $V=v/\lambda$. Then $-\Delta_p V=\lambda^{1-p}g$ weakly in $B_2$, with $\|V\|_{L^\infty(B_2)}\le 1$ and $\|\lambda^{1-p}g\|_{L^\infty(B_2)}\le 1$. The standard interior $C^{1,\beta}$ estimates for degenerate quasilinear equations, as developed for instance by DiBenedetto \cite{DiBenedetto1983} and Tolksdorf \cite{Tolksdorf1984}, give $\sup_{B_1}|\nabla V|\le C(N,p)$. Multiplying by $\lambda$ gives the claim.
\end{proof}

Finally we record two facts about superposition that are used in Sections \ref{sec_sing} and \ref{sec_index}. When $p>2$ we write
\begin{equation}\label{eq_Mdef}
\mathcal M(\xi)=|\xi|^{p-2}\mathrm{Id}+(p-2)|\xi|^{p-4}\xi\otimes\xi\quad(\xi\ne0),
\qquad \mathcal M(0)=0
\end{equation}
for the matrix that linearizes the $p$-Laplacian; this is the matrix implicit in the first two terms of $Q_u$, since $\mathcal M(\xi)\zeta\cdot\zeta=|\xi|^{p-2}|\zeta|^2+(p-2)|\xi|^{p-4}(\xi\cdot\zeta)^2$.

\begin{lemma}\label{lem_nem}
Let $\Omega\subseteq\mathbb{R}^N$ be open.
\begin{enumerate}[label=\rm(\roman*)]
\item For every $q>1$ the map $z\mapsto|z|^{q-1}$ acts continuously from $L^{q+1}(\Omega)$ into $L^{\frac{q+1}{q-1}}(\Omega)$.
\item Let $p>2$. Then $\xi\mapsto|\xi|^{p-2}\xi$ is of class $C^1$ on $\mathbb{R}^N$ with differential $\mathcal M$,
\begin{equation}\label{eq_Mbound}
|\mathcal M(\xi)|\le(p-1)|\xi|^{p-2}\qquad\text{for every }\xi\in\mathbb{R}^N ,
\end{equation}
and $\mathcal M$ acts continuously from $L^p(\Omega;\mathbb{R}^N)$ into $L^{\frac{p}{p-2}}(\Omega;\mathbb{R}^{N\times N})$.
\end{enumerate}
\end{lemma}

\begin{proof}
We first dispose of the pointwise statements in (ii). Away from the origin $\xi\mapsto|\xi|^{p-2}\xi$ is differentiable with differential $\mathcal M(\xi)$, and for a unit vector $e$ we have $|\mathcal M(\xi)e|\le|\xi|^{p-2}+(p-2)|\xi|^{p-2}$, which is \eqref{eq_Mbound}. At the origin, $\big||\xi|^{p-2}\xi\big|=|\xi|^{p-1}=o(|\xi|)$ because $p>2$, so the map is differentiable there with differential $0=\mathcal M(0)$; and the right hand side of \eqref{eq_Mbound} tends to $0$ as $\xi\to0$, so $\mathcal M$ is continuous at the origin. Hence the map is of class $C^1$ on $\mathbb{R}^N$.

Both superposition statements are instances of the following observation: if $F$ is a continuous map from $\mathbb{R}^d$ into a finite dimensional normed space with $|F(z)|\le C_F|z|^{a}$ for every $z$, and if $0<a\le b<\infty$ with $b\ge1$, then $F$ acts continuously from $L^{b}(\Omega;\mathbb{R}^d)$ into $L^{b/a}(\Omega)$. The growth bound gives $|F(z)|^{b/a}\le C_F^{b/a}|z|^{b}$, so $F$ does map $L^b$ into $L^{b/a}$. Let $z_j\to z$ in $L^b$. It is enough to produce, in every subsequence, a further subsequence along which $F(z_j)\to F(z)$ in $L^{b/a}$. Passing to a subsequence we may assume $\|z_{j+1}-z_j\|_{L^b}\le2^{-j}$, so that
\[
h:=|z_1|+\sum_{j\ge1}|z_{j+1}-z_j|
\]
satisfies $\|h\|_{L^b}\le\|z_1\|_{L^b}+1$ and $|z_j|\le h$ for every $j$. The series also converges almost everywhere, so $z_j\to z$ almost everywhere along this subsequence and $|z|\le h$ almost everywhere. By the continuity of $F$ we get $|F(z_j)-F(z)|^{b/a}\to0$ almost everywhere, while
\[
\big|F(z_j)-F(z)\big|^{b/a}\le\big(2C_Fh^{a}\big)^{b/a}=(2C_F)^{b/a}h^{b}\in L^1(\Omega),
\]
and dominated convergence concludes. Part (i) is the case $(a,b)=(q-1,q+1)$, $F(z)=|z|^{q-1}$, and the last assertion of (ii) is the case $(a,b)=(p-2,p)$, $F=\mathcal M$, the growth bound being \eqref{eq_Mbound}.
\end{proof}

\section{Integral estimates from stability}\label{sec_stab}

Throughout this section $N>p\ge 2$, $q>p-1$, and $u\in C^1(\mathbb{R}^N)$ is a weak solution of \eqref{eq_main} which is stable outside a compact set $K$. We fix $R_0\ge 1$ with $K\subset B_{R_0}$, and we set
\begin{equation}
\delta_0=\frac{2q-(p-1)+2\sqrt{q(q-p+1)}}{p-1},
\end{equation}
which is the larger root of the quadratic $(p-1)\delta^2-\big(4q-2(p-1)\big)\delta+(p-1)=0$ and which satisfies $\delta_0>1$ because $q>p-1$.

The following estimate is not new. For $p>2$ it is \cite[Prop.~1.4]{DFSV2009}, and for $p=2$ it goes back to \cite{Farina2007}. We give the short proof because we use it for every $p\ge2$ in a single statement, and because the exact range of admissible $\delta$ is used quantitatively in Lemma \ref{lem_qc} and in Proposition \ref{prop_decay}.

\begin{lemma}[{\cite[Prop.~1.4]{DFSV2009}}]\label{lem_stab}
Let $\delta\in[1,\delta_0)$ and let $\gamma_0\ge \frac{p(q+\delta)}{q-p+1}$. There is $C=C(N,p,q,\delta,\gamma_0)$ such that
\begin{equation}
\int_{\mathbb{R}^N}|u|^{q+\delta}\eta^{\gamma_0}\,dx+\int_{\mathbb{R}^N}|u|^{\delta-1}|\nabla u|^p\eta^{\gamma_0}\,dx
\le C\int_{\mathbb{R}^N}|\nabla\eta|^{\frac{p(q+\delta)}{q-p+1}}\,dx
\end{equation}
for every $\eta\in C_c^\infty(\mathbb{R}^N\setminus K)$ with $0\le\eta\le 1$.
\end{lemma}

\begin{proof}
Since $\delta\ge 1$ and $u\in C^1(\mathbb{R}^N)$, the function $|u|^{\delta-1}u\,\eta^{\gamma_0}$ belongs to $W^{1,p}(\mathbb{R}^N)$ and has compact support in $\mathbb{R}^N\setminus K$. Every integral written below is finite, because $u$ and $\nabla u$ are continuous and $\eta$ has compact support, so each of the absorption steps is legitimate. Testing \eqref{eq_main} with it gives
\begin{equation}\label{eq_test1}
\int|u|^{q+\delta}\eta^{\gamma_0}
=\delta\int|u|^{\delta-1}|\nabla u|^p\eta^{\gamma_0}
+\gamma_0\int|u|^{\delta-1}u\,|\nabla u|^{p-2}\nabla u\cdot\nabla\eta\ \eta^{\gamma_0-1},
\end{equation}
where all integrals are over $\mathbb{R}^N$. For $\varepsilon\in(0,1)$, Young's inequality with exponents $\frac{p}{p-1}$ and $p$ applied to the factorization
\[
|u|^{\delta-1}|u|\,|\nabla u|^{p-1}\eta^{\gamma_0-1}|\nabla\eta|
=\left(|u|^{\frac{(\delta-1)(p-1)}{p}}|\nabla u|^{p-1}\eta^{\frac{\gamma_0(p-1)}{p}}\right)
\left(|u|^{\frac{\delta-1}{p}+1}\eta^{\frac{\gamma_0}{p}-1}|\nabla\eta|\right)
\]
yields
\begin{equation}\label{eq_young1}
\gamma_0\left|\int|u|^{\delta-1}u\,|\nabla u|^{p-2}\nabla u\cdot\nabla\eta\ \eta^{\gamma_0-1}\right|
\le \varepsilon\int|u|^{\delta-1}|\nabla u|^p\eta^{\gamma_0}
+C_\varepsilon\int|u|^{\delta+p-1}\eta^{\gamma_0-p}|\nabla\eta|^p .
\end{equation}
Next we use stability. The function $\varphi=|u|^{\frac{\delta-1}{2}}u\,\eta^{\gamma_0/2}$ belongs to $C_c^1(\mathbb{R}^N\setminus K)$, so $Q_u(\varphi)\ge 0$. Since $p\ge 2$ we have $|\nabla u|^{p-4}(\nabla u\cdot\nabla\varphi)^2\le|\nabla u|^{p-2}|\nabla\varphi|^2$ pointwise, whence
\[
q\int|u|^{q+\delta}\eta^{\gamma_0}\le (p-1)\int|\nabla u|^{p-2}|\nabla\varphi|^2 .
\]
Expanding $\nabla\varphi$ and applying Young's inequality in the form $|a+b|^2\le(1+\varepsilon)a^2+(1+\varepsilon^{-1})b^2$,
\[
|\nabla\varphi|^2\le(1+\varepsilon)\frac{(\delta+1)^2}{4}|u|^{\delta-1}|\nabla u|^2\eta^{\gamma_0}
+C_\varepsilon|u|^{\delta+1}\eta^{\gamma_0-2}|\nabla\eta|^2 ,
\]
so that
\begin{equation}\label{eq_stab1}
q\int|u|^{q+\delta}\eta^{\gamma_0}
\le (p-1)(1+\varepsilon)\frac{(\delta+1)^2}{4}\int|u|^{\delta-1}|\nabla u|^p\eta^{\gamma_0}
+C_\varepsilon\int|u|^{\delta+1}|\nabla u|^{p-2}\eta^{\gamma_0-2}|\nabla\eta|^2 .
\end{equation}
If $p>2$, Young's inequality with exponents $\frac{p}{p-2}$ and $\frac p2$ applied to
\[
|u|^{\delta+1}|\nabla u|^{p-2}\eta^{\gamma_0-2}|\nabla\eta|^2
=\left(|u|^{\frac{(\delta-1)(p-2)}{p}}|\nabla u|^{p-2}\eta^{\frac{\gamma_0(p-2)}{p}}\right)
\left(|u|^{\delta+1-\frac{(\delta-1)(p-2)}{p}}\eta^{\frac{2\gamma_0}{p}-2}|\nabla\eta|^2\right)
\]
gives, since $\frac p2\left(\delta+1-\frac{(\delta-1)(p-2)}{p}\right)=\delta+p-1$,
\begin{equation}\label{eq_young2}
C_\varepsilon\int|u|^{\delta+1}|\nabla u|^{p-2}\eta^{\gamma_0-2}|\nabla\eta|^2
\le\varepsilon\int|u|^{\delta-1}|\nabla u|^p\eta^{\gamma_0}
+C\int|u|^{\delta+p-1}\eta^{\gamma_0-p}|\nabla\eta|^p ,
\end{equation}
where $C$ depends on $\varepsilon$. For $p=2$ the inequality \eqref{eq_young2} holds trivially with the first term on the right omitted.

From \eqref{eq_test1} and \eqref{eq_young1},
\begin{equation}\label{eq_grad_control}
\int|u|^{\delta-1}|\nabla u|^p\eta^{\gamma_0}
\le\frac{1}{\delta-\varepsilon}\left(\int|u|^{q+\delta}\eta^{\gamma_0}
+C_\varepsilon\int|u|^{\delta+p-1}\eta^{\gamma_0-p}|\nabla\eta|^p\right).
\end{equation}
Substituting \eqref{eq_young2} and \eqref{eq_grad_control} into \eqref{eq_stab1} we obtain
\[
\left(q-\frac{(p-1)(1+\varepsilon)\frac{(\delta+1)^2}{4}+\varepsilon}{\delta-\varepsilon}\right)
\int|u|^{q+\delta}\eta^{\gamma_0}\le C\int|u|^{\delta+p-1}\eta^{\gamma_0-p}|\nabla\eta|^p .
\]
The condition $q\delta>(p-1)\frac{(\delta+1)^2}{4}$ is equivalent to $(p-1)\delta^2-\big(4q-2(p-1)\big)\delta+(p-1)<0$, whose roots are
\[
\delta_\pm=\frac{2q-(p-1)\pm 2\sqrt{q(q-p+1)}}{p-1},
\]
so it holds precisely for $\delta\in(\delta_-,\delta_+)=(\delta_-,\delta_0)$. Since $q>p-1$ the value $\delta=1$ satisfies it strictly, so $\delta_-<1$ and the whole interval $[1,\delta_0)$ is allowed. Write $\kappa_\delta=q-\frac{(p-1)(\delta+1)^2}{4\delta}$, which is positive by the previous line. The coefficient on the left is a continuous function of $\varepsilon$ on $[0,\delta/2]$ whose value at $\varepsilon=0$ is $\kappa_\delta$, so there is $\varepsilon_\delta\in(0,\delta/2]$ such that this coefficient is at least $\kappa_\delta/2$ for every $\varepsilon\in(0,\varepsilon_\delta]$. Fixing such an $\varepsilon$ and using \eqref{eq_grad_control} we get
\begin{equation}\label{eq_almost}
\int|u|^{q+\delta}\eta^{\gamma_0}+\int|u|^{\delta-1}|\nabla u|^p\eta^{\gamma_0}
\le C\int|u|^{\delta+p-1}\eta^{\gamma_0-p}|\nabla\eta|^p .
\end{equation}
Finally, Young's inequality with exponents $\frac{q+\delta}{\delta+p-1}$ and $\frac{q+\delta}{q-p+1}$ applied to the factorization
\[
|u|^{\delta+p-1}\eta^{\gamma_0-p}|\nabla\eta|^p
=\left(|u|^{\delta+p-1}\eta^{\frac{\gamma_0(\delta+p-1)}{q+\delta}}\right)
\left(\eta^{\gamma_0-p-\frac{\gamma_0(\delta+p-1)}{q+\delta}}|\nabla\eta|^p\right)
\]
gives
\[
C\int|u|^{\delta+p-1}\eta^{\gamma_0-p}|\nabla\eta|^p
\le\frac12\int|u|^{q+\delta}\eta^{\gamma_0}
+C\int\eta^{\gamma_0-\frac{p(q+\delta)}{q-p+1}}|\nabla\eta|^{\frac{p(q+\delta)}{q-p+1}} .
\]
The exponent of $\eta$ in the last integral equals $\gamma_0-\frac{p(q+\delta)}{q-p+1}\ge 0$ by the choice of $\gamma_0$, so that factor is bounded by $1$. Absorbing the first term on the right into the left hand side of \eqref{eq_almost} completes the proof.
\end{proof}

\begin{corollary}\label{cor_annuli}
Let $\delta\in[1,\delta_0)$ and write $m_\delta=\frac{p(q+\delta)}{q-p+1}$. There is $C=C(N,p,q,\delta,R_0)$ such that for every $r\ge 8R_0$,
\begin{equation}\label{eq_ann}
\int_{B_{4r}\setminus B_r}|u|^{q+\delta}\,dx\le C\,r^{\,N-m_\delta},
\end{equation}
\begin{equation}\label{eq_ball}
\int_{B_{r}\setminus B_{2R_0}}|u|^{q+\delta}\,dx\le C\left(1+r^{\,N-m_\delta}\right).
\end{equation}
\end{corollary}

\begin{proof}
For \eqref{eq_ann}, choose $\eta\in C_c^\infty(B_{8r}\setminus \overline{B_{r/2}})$ with $0\le\eta\le1$, $\eta\equiv1$ on $B_{4r}\setminus B_r$ and $|\nabla\eta|\le C/r$. This is admissible because $r/2\ge 4R_0>R_0$. Then $\int|\nabla\eta|^{m_\delta}\le C r^{-m_\delta}|B_{8r}|\le Cr^{N-m_\delta}$, and Lemma \ref{lem_stab} gives \eqref{eq_ann}.

For \eqref{eq_ball}, choose $\eta\in C_c^\infty(B_{2r}\setminus\overline{B_{R_0}})$ with $\eta\equiv1$ on $B_r\setminus B_{2R_0}$, $|\nabla\eta|\le C/R_0$ on $B_{2R_0}\setminus B_{R_0}$ and $|\nabla\eta|\le C/r$ on $B_{2r}\setminus B_r$. The inner annulus contributes at most $CR_0^{N-m_\delta}\le C$ and the outer one at most $Cr^{N-m_\delta}$.
\end{proof}

The next lemma is where the Joseph--Lundgren exponent appears. It says that the estimates of Corollary \ref{cor_annuli} reach the scaling exponent $\frac{N(q-p+1)}{p}$ exactly when $q<q_c(N,p)$.

\begin{lemma}\label{lem_qc}
Let $N>p>1$ and define, for $q>p-1$,
\[
G(q)=q+\delta_0(q)-\frac{N(q-p+1)}{p},
\qquad \delta_0(q)=\frac{2q-(p-1)+2\sqrt{q(q-p+1)}}{p-1}.
\]
Then $G$ is strictly concave on $(p-1,+\infty)$ and $\lim_{q\to(p-1)^+}G(q)=p>0$. Moreover:
\begin{enumerate}[label=(\roman*)]
\item if $(p-1)N\le p(p+3)$, then $G(q)>0$ for every $q>p-1$;
\item if $(p-1)N> p(p+3)$, then $G$ has a unique zero in $(p-1,+\infty)$, it is given by the formula \eqref{eq_qc} for $q_c(N,p)$, and $G(q)>0$ for $p-1<q<q_c(N,p)$ while $G(q)<0$ for $q>q_c(N,p)$.
\end{enumerate}
\end{lemma}

\begin{proof}
Differentiating,
\[
G'(q)=1+\frac{2+h(q)}{p-1}-\frac{N}{p},
\qquad h(q)=\frac{2q-(p-1)}{\sqrt{q(q-p+1)}} .
\]
Since
\[
h(q)^2=\frac{4q^2-4(p-1)q+(p-1)^2}{q^2-(p-1)q}=4+\frac{(p-1)^2}{q^2-(p-1)q}
\]
and $q\mapsto q^2-(p-1)q$ is strictly increasing on $(p-1,+\infty)$, the function $h$ is strictly decreasing there, with $h(q)\to+\infty$ as $q\to(p-1)^+$ and $h(q)\to 2$ as $q\to+\infty$. Hence $G'$ is strictly decreasing, so $G$ is strictly concave, and
\[
\lim_{q\to+\infty}G'(q)=1+\frac{4}{p-1}-\frac Np=\frac{p+3}{p-1}-\frac Np .
\]
Also $G(p-1^+)=(p-1)+\frac{2(p-1)-(p-1)}{p-1}-0=p$.

If $(p-1)N\le p(p+3)$, then $\lim_{q\to\infty}G'\ge 0$, and since $G'$ is strictly decreasing with this limit we get $G'>0$ on $(p-1,+\infty)$. Hence $G$ is increasing and $G>p>0$, which is (i).

If $(p-1)N>p(p+3)$, then $\lim_{q\to\infty}G'<0$, so $G(q)\to-\infty$. Being concave with $G(p-1^+)=p>0$, the function $G$ has exactly one zero $q_c$ in $(p-1,+\infty)$, is positive to its left and negative to its right.

It remains to identify $q_c$. Put $B=(p-1)N-p$, $a=B-p^2$ and $b=(p-1)B$. In case (ii) we have $(p-1)N>p(p+3)>p$, so $B>0$ and $b>0$; also $a-2p=(p-1)N-p(p+3)>0$, so $a>2p>0$. A computation gives
\[
(p-1)\left(\frac{N(q-p+1)}{p}-q\right)-\big(2q-(p-1)\big)=\frac{aq-b}{p},
\]
so $G(q)=0$ is equivalent to
\begin{equation}\label{eq_root}
\frac{aq-b}{p}=2\sqrt{q(q-p+1)} .
\end{equation}
In particular $aq-b\ge 0$ at the zero, and squaring \eqref{eq_root} shows that $q_c$ is a root of
\[
P(q):=(a^2-4p^2)q^2+\big(4p^2(p-1)-2ab\big)q+b^2 .
\]
Here
\begin{align*}
a^2-4p^2
&=(a-2p)(a+2p)\\
&=\big((p-1)N-p(p+3)\big)\big((p-1)N-p(p-1)\big)\\
&=(p-1)(N-p)\big((p-1)N-p(p+3)\big)>0.
\end{align*}
and, using $a=B-p^2$ and $b=(p-1)B$,
\[
\left(\frac{4p^2(p-1)-2ab}{2}\right)^2-(a^2-4p^2)b^2
=4p^2\left(b^2-(p-1)ab+p^2(p-1)^2\right)
=4p^2(p-1)^2\left(B^2-aB+p^2\right).
\]
Since $a=B-p^2$ we have $B^2-aB+p^2=p^2B+p^2=p^2\big((p-1)N-p+1\big)=p^2(p-1)(N-1)$, so the quantity above equals $4p^4(p-1)^3(N-1)$ and its square root is $2p^2(p-1)\sqrt{(p-1)(N-1)}$. Also $ab=(p-1)(B^2-p^2B)$, hence
\[
\frac{2ab-4p^2(p-1)}{2}+2p^2(p-1)\sqrt{(p-1)(N-1)}
=(p-1)\left(B^2-p^2B-2p^2+2p^2\sqrt{(p-1)(N-1)}\right).
\]
Therefore the larger root of $P$ is
\[
\frac{B^2-p^2B-2p^2+2p^2\sqrt{(p-1)(N-1)}}{(N-p)\big((p-1)N-p(p+3)\big)} .
\]
Substituting $B=(p-1)N-p$ in $-p^2B-2p^2=-p^2(p-1)N+p^3-2p^2=-p^2(p-1)N+p^2(p-2)$ shows that this expression coincides with the right hand side of \eqref{eq_qc}. It is the larger root that corresponds to $q_c$. Indeed, \eqref{eq_root} gives $aq_c-b=2p\sqrt{q_c(q_c-p+1)}>0$, so $q_c>b/a$, while a direct computation using $(p-1)a-b=(p-1)(B-p^2)-(p-1)B=-(p-1)p^2$ gives
\[
P\!\left(\frac ba\right)=\frac{4p^2b}{a^2}\big((p-1)a-b\big)=-\frac{4p^4(p-1)b}{a^2}<0 .
\]
Since $P$ has positive leading coefficient, the point $b/a$ lies strictly between the two roots of $P$, so any root of $P$ larger than $b/a$ is the larger one. This proves (ii).
\end{proof}

\begin{remark}\label{rem_qcbig}
Writing $\delta_1=\frac{N(q-p+1)}{p}-q$ one has $G(q)=\delta_0(q)-\delta_1(q)$. At $q=p^*-1$ we have $\delta_1=1<\delta_0$, so $G(p^*-1)>0$ and therefore $q_c(N,p)>p^*-1$, as announced in Section \ref{sec_intro}.
\end{remark}

\begin{remark}
Lemma \ref{lem_qc} is implicit in \cite{DFSV2009}, where $q_c(N,p)$ arises as the threshold beyond which the stability estimates stop giving a negative power of $r$. We have included the computation because the identification of the threshold with the closed formula \eqref{eq_qc} is used quantitatively below.
\end{remark}

\section{Decay at the self-similar rate}\label{sec_decay}

The first limit in the next statement is \cite[Thm.~1.10]{DFSV2009} when $p>2$. We prove it for every $p\ge2$, and we prove at the same time the gradient decay, which is not stated there and which Section \ref{sec_boot} needs.

\begin{proposition}[{compare \cite[Thm.~1.10]{DFSV2009}}]\label{prop_decay}
Let $N>p\ge 2$ and $p^*-1\le q<q_c(N,p)$, and let $u\in C^1(\mathbb{R}^N)$ be a weak solution of \eqref{eq_main} that is stable outside a compact set. Then
\begin{equation}\label{eq_decay}
\lim_{|x|\to+\infty}|x|^{\alpha}|u(x)|=0
\qquad\text{and}\qquad
\lim_{|x|\to+\infty}|x|^{\alpha+1}|\nabla u(x)|=0 .
\end{equation}
\end{proposition}

\begin{proof}
We keep the notation of Section \ref{sec_stab}: $K$ is a compact set outside which $u$ is stable, and $R_0\ge1$ satisfies $K\subset B_{R_0}$. Put
\[
s_1=\frac{N(q-p+1)}{p},\qquad \delta_1=s_1-q .
\]
The inequality $\delta_1\ge 1$ is equivalent to $N(q-p+1)\ge p(q+1)$, that is to $q(N-p)\ge N(p-1)+p$, that is to $q\ge p^*-1$. The inequality $\delta_1<\delta_0$ is exactly $G(q)>0$ in the notation of Lemma \ref{lem_qc}, hence it holds because $q<q_c(N,p)$. Therefore $\delta_1\in[1,\delta_0)$, and the corresponding exponent in Corollary \ref{cor_annuli} is
\[
N-\frac{p(q+\delta_1)}{q-p+1}=N-\frac{p s_1}{q-p+1}=0 .
\]
By \eqref{eq_ball},
\begin{equation}\label{eq_Ls1}
\int_{\mathbb{R}^N\setminus B_{2R_0}}|u|^{s_1}\,dx\le C<+\infty .
\end{equation}
Since $\delta_1<\delta_0$ we may fix $\bar\delta\in(\delta_1,\delta_0)$ and set
\[
\eta_0=\frac{p(q+\bar\delta)}{q-p+1}-N>0 .
\]
By \eqref{eq_ann}, for every $r\ge 8R_0$,
\begin{equation}\label{eq_Lsbar}
\int_{B_{4r}\setminus B_r}|u|^{q+\bar\delta}\,dx\le C\,r^{-\eta_0} .
\end{equation}

For $r\ge 8R_0$ define $u_r(x)=r^{\alpha}u(rx)$ for $x\in B_4\setminus\overline{B_1}$. A change of variables shows that $u_r\in C^1$ is a weak solution of $-\Delta_p u_r=|u_r|^{q-1}u_r$ there. Since $\alpha(q+\bar\delta)=\frac{p(q+\bar\delta)}{q-p+1}=N+\eta_0$ and $\alpha s_1=N$, we get from \eqref{eq_Lsbar} and \eqref{eq_Ls1}
\begin{equation}\label{eq_scaled}
\int_{B_4\setminus B_1}|u_r|^{q+\bar\delta}\,dx=r^{\alpha(q+\bar\delta)-N}\int_{B_{4r}\setminus B_r}|u|^{q+\bar\delta}\,dx\le C,
\end{equation}
\begin{equation}
\int_{B_4\setminus B_1}|u_r|^{s_1}\,dx=\int_{B_{4r}\setminus B_r}|u|^{s_1}\,dx=:\omega(r)^{s_1},
\end{equation}
and $\omega(r)\to 0$ as $r\to+\infty$ because the integral in \eqref{eq_Ls1} converges and the annuli $B_{4r}\setminus B_r$ escape to infinity.

Write the equation for $u_r$ as $-\Delta_p u_r=c_r(x)|u_r|^{p-2}u_r$ with $c_r=|u_r|^{q-p+1}$, and set
\[
k=\frac{q+\bar\delta}{q-p+1}>\frac{q+\delta_1}{q-p+1}=\frac{s_1}{q-p+1}=\frac{N}{p} .
\]
Then $\|c_r\|_{L^k(B_4\setminus B_1)}^{k}=\int_{B_4\setminus B_1}|u_r|^{q+\bar\delta}\le C$ by \eqref{eq_scaled}, uniformly in $r$. Lemma \ref{lem_serrin} therefore gives a constant $C$ independent of $r$ with $\sup_{B_3\setminus B_2}|u_r|\le C\|u_r\|_{L^p(B_4\setminus B_1)}$. Since $q\ge p^*-1$ we have $s_1\ge\frac Np\cdot\frac{p^2}{N-p}=p^*>p$, so H\"older's inequality on the bounded set $B_4\setminus B_1$ turns this into
\[
\sup_{B_3\setminus B_2}|u_r|\le C\,\|u_r\|_{L^{s_1}(B_4\setminus B_1)}=C\,\omega(r) .
\]
Undoing the scaling, $\sup_{2r\le|x|\le 3r}|x|^{\alpha}|u(x)|\le 3^{\alpha}\sup_{B_3\setminus B_2}|u_r|\le C\omega(r)\to0$, and since the annuli $\{2r\le|x|\le3r\}$ cover a neighborhood of infinity as $r$ ranges over $[8R_0,\infty)$, the first limit in \eqref{eq_decay} follows.

For the gradient, set $M_r=\sup_{2r\le|x|\le 3r}|u|$, so that $r^{\alpha}M_r\to0$. Let $x_0$ with $|x_0|=\frac52 r$ and apply Lemma \ref{lem_grad} to $v(y)=u(x_0+\frac r4 y)$ on $B_2$, which solves $-\Delta_p v=(r/4)^p|v|^{q-1}v$. We obtain
\[
\frac r4\,|\nabla u(x_0)|\le C\left(M_r+\big((r/4)^pM_r^q\big)^{\frac1{p-1}}\right).
\]
Since $r^{\alpha}M_r\to 0$ we have
\[
r^pM_r^{q-p+1}=\big(r^{\alpha}M_r\big)^{q-p+1}\longrightarrow0.
\]
Hence, for $r$ large,
\[
\big(r^pM_r^q\big)^{1/(p-1)}
=\big(r^pM_r^{q-p+1}\big)^{1/(p-1)}M_r\le M_r.
\]
Therefore $|\nabla u(x_0)|\le CM_r/r$, and $|x_0|^{\alpha+1}|\nabla u(x_0)|\le C r^{\alpha}M_r\to0$. This proves the second limit in \eqref{eq_decay}.
\end{proof}

\section{From self-similar decay to \texorpdfstring{$p$-harmonic}{p-harmonic} decay}\label{sec_boot}

This section contains the new ingredient. Recall $\alpha=\frac{p}{q-p+1}$ and $\gamma=\frac{N-p}{p-1}$, and note that
\begin{equation}\label{eq_alphagamma}
\alpha<\gamma \iff p(p-1)<(N-p)(q-p+1) \iff q>\frac{N(p-1)}{N-p},
\end{equation}
which holds in particular when $q>p^*-1$.

\begin{lemma}\label{lem_super}
Let $0<\alpha<\gamma$ with $\gamma=\frac{N-p}{p-1}$, let $A,B>0$ and set $w(x)=A|x|^{-\alpha}+B|x|^{-\gamma}$ on $\mathbb{R}^N\setminus\{0\}$. Then $w$ is smooth there and
\begin{equation}\label{eq_superid}
-\Delta_p w(x)=\left(1+\frac{\gamma B}{\alpha A}|x|^{\alpha-\gamma}\right)^{p-2}
K\,A^{p-1}\,|x|^{-1-(\alpha+1)(p-1)},
\qquad K:=\alpha^{p-1}(p-1)(\gamma-\alpha)>0 .
\end{equation}
If moreover $p\ge2$ and $\alpha=\frac{p}{q-p+1}$, then $1+(\alpha+1)(p-1)=\alpha q$ and
\begin{equation}\label{eq_superineq}
-\Delta_p w(x)\ge K\,A^{p-1}\,|x|^{-\alpha q}\qquad\text{for } x\ne0 .
\end{equation}
\end{lemma}

\begin{proof}
Write $r=|x|$. Then $w'(r)<0$ and
\[
-w'(r)=\alpha A r^{-\alpha-1}+\gamma B r^{-\gamma-1}=\alpha A r^{-\alpha-1}(1+t(r)),
\qquad t(r)=\frac{\gamma B}{\alpha A}r^{\alpha-\gamma}.
\]
For radial functions, $-\Delta_p w=r^{1-N}\frac{d}{dr}\big[r^{N-1}(-w')^{p-1}\big]$. Set
\[
m=N-1-(\alpha+1)(p-1).
\]
Since $(p-1)\gamma=N-p$, we have $m=N-1-(p-1)-\alpha(p-1)=(N-p)-\alpha(p-1)=(p-1)(\gamma-\alpha)>0$. Next,
\[
r^{N-1}(-w')^{p-1}=(\alpha A)^{p-1}r^{m}(1+t)^{p-1},
\qquad t'(r)=(\alpha-\gamma)\frac{t}{r}=-\frac{m}{p-1}\cdot\frac tr .
\]
Hence
\[
\frac{d}{dr}\Big[r^{N-1}(-w')^{p-1}\Big]
=(\alpha A)^{p-1}\left[m r^{m-1}(1+t)^{p-1}+r^{m}(p-1)(1+t)^{p-2}t'\right]
=m(\alpha A)^{p-1}r^{m-1}(1+t)^{p-2},
\]
where we used $(1+t)^{p-1}-\frac{m}{p-1}\cdot\frac{(p-1)t}{m}(1+t)^{p-2}=(1+t)^{p-2}\big[(1+t)-t\big]=(1+t)^{p-2}$. Since $m-N=-1-(\alpha+1)(p-1)$ and $m=(p-1)(\gamma-\alpha)$, this gives \eqref{eq_superid}.

Finally, $1+(\alpha+1)(p-1)=p+\alpha(p-1)$ and $\alpha q=\alpha(q-p+1)+\alpha(p-1)=p+\alpha(p-1)$, so the two exponents agree. For $p\ge2$ and $t>0$ we have $(1+t)^{p-2}\ge1$, which gives \eqref{eq_superineq}.
\end{proof}

\begin{proposition}\label{prop_boot}
Let $N>p\ge2$ and $\frac{N(p-1)}{N-p}<q$, and set
\[
\kappa=\frac{q}{p-1}>1,\qquad \theta=\frac{\log 2}{\log\kappa} .
\]
There is $\varepsilon_*=\varepsilon_*(N,p,q)>0$ with the following property. If $u\in C^1(\mathbb{R}^N)$ is a weak solution of \eqref{eq_main} with
\begin{equation}\label{eq_smalldecay}
\limsup_{|x|\to\infty}|x|^\alpha|u(x)|<\varepsilon_* ,
\end{equation}
then there are $C>0$ and $R_*>1$ such that
\begin{equation}\label{eq_pharmonic}
|u(x)|\le C\,|x|^{-\gamma}\left(\log|x|\right)^{\theta}\qquad\text{for } |x|\ge R_* .
\end{equation}
\end{proposition}

\begin{proof}
Let $K=\alpha^{p-1}(p-1)(\gamma-\alpha)$ be the constant of Lemma \ref{lem_super}, positive by \eqref{eq_alphagamma}, and put
\[
c_1=\log\left(2K^{-\frac{1}{p-1}}\right),\qquad \bar c=\frac{c_1}{\kappa-1} .
\]
The two conditions
\begin{equation}\label{eq_eps0}
2K^{-\frac1{p-1}}\varepsilon^{\kappa-1}\le\frac12
\qquad\text{and}\qquad -\log\varepsilon\ge \bar c+1
\end{equation}
both hold for all small $\varepsilon>0$, and $K$, $\kappa$ and $\bar c$ depend only on $N$, $p$ and $q$; let $\varepsilon_*=\varepsilon_*(N,p,q)\in(0,1)$ be a value such that \eqref{eq_eps0} holds for every $\varepsilon\in(0,\varepsilon_*]$. By \eqref{eq_smalldecay} we may fix $\varepsilon_0\in(0,\varepsilon_*]$ with \eqref{eq_eps0} and then $R_0\ge3$ with
\[
|u(x)|\le\varepsilon_0|x|^{-\alpha}\qquad\text{for } |x|\ge R_0 .
\] Define recursively
\begin{equation}\label{eq_rec}
A_n=K^{-\frac{1}{p-1}}\varepsilon_n^{\kappa},\qquad
\varepsilon_{n+1}=2A_n,\qquad
R_{n+1}=\left(\frac{\varepsilon_n}{A_n}\right)^{\frac{1}{\gamma-\alpha}}R_n .
\end{equation}
By \eqref{eq_eps0} and induction, $\varepsilon_{n+1}\le\frac12\varepsilon_n$ for all $n$, so $\varepsilon_n\le 2^{-n}\varepsilon_0$ decreases to $0$. Consequently $\varepsilon_n/A_n=2\varepsilon_n/\varepsilon_{n+1}\ge4$, so $R_{n+1}\ge 4^{1/(\gamma-\alpha)}R_n$ and $R_n$ increases to $+\infty$.

\smallskip
\emph{Step 1: the inductive comparison.} We claim that
\begin{equation}\label{eq_induct}
|u(x)|\le\varepsilon_n|x|^{-\alpha}\quad\text{for }|x|\ge R_n
\end{equation}
holds for every $n\ge0$. It holds for $n=0$. Assume it for some $n$ and set
\[
B_n=\varepsilon_n R_n^{\gamma-\alpha},\qquad w_n(x)=A_n|x|^{-\alpha}+B_n|x|^{-\gamma} .
\]
By Lemma \ref{lem_super} and the definition of $A_n$,
\[
-\Delta_p w_n\ge K A_n^{p-1}|x|^{-\alpha q}=\varepsilon_n^{q}|x|^{-\alpha q}\ge|u(x)|^{q}
\qquad\text{for }|x|>R_n ,
\]
where the last inequality uses \eqref{eq_induct}. By Lemma \ref{lem_abs}, $|u|$ is a weak subsolution of $-\Delta_p z=|u|^q$ in $\{|x|>R_n\}$. On $\partial B_{R_n}$,
\[
w_n\ge B_nR_n^{-\gamma}=\varepsilon_nR_n^{-\alpha}\ge|u| ,
\]
again by \eqref{eq_induct}. Finally $|u(x)|\to0$ and $w_n(x)\to0$ as $|x|\to\infty$. Lemma \ref{lem_comp}, applied with $v=|u|$, $w=w_n$, $f=|u|^q$ and $R=R_n$, gives
\begin{equation}\label{eq_compconc}
|u(x)|\le A_n|x|^{-\alpha}+B_n|x|^{-\gamma}\qquad\text{for }|x|\ge R_n .
\end{equation}
If $|x|\ge R_{n+1}$, then $|x|^{\gamma-\alpha}\ge\frac{\varepsilon_n}{A_n}R_n^{\gamma-\alpha}$, that is $B_n|x|^{-\gamma}\le A_n|x|^{-\alpha}$, so \eqref{eq_compconc} yields $|u(x)|\le 2A_n|x|^{-\alpha}=\varepsilon_{n+1}|x|^{-\alpha}$. This is \eqref{eq_induct} for $n+1$.

\smallskip
\emph{Step 2: the bookkeeping.} Taking logarithms in \eqref{eq_rec},
\[
\log R_{n+1}-\log R_n=\frac{\log2+\log\varepsilon_n-\log\varepsilon_{n+1}}{\gamma-\alpha},
\]
and summing from $0$ to $n-1$,
\[
(\gamma-\alpha)\big(\log R_n-\log R_0\big)=n\log 2+\log\varepsilon_0-\log\varepsilon_n .
\]
Exponentiating gives the identity
\begin{equation}\label{eq_identity}
\varepsilon_nR_n^{\gamma-\alpha}=\varepsilon_0R_0^{\gamma-\alpha}\,2^{\,n}=:c_0\,2^{\,n} ,
\end{equation}
so that $B_n=c_02^n$ and, since $A_n=\frac12\varepsilon_{n+1}$,
\begin{equation}\label{eq_An}
A_nR_{n+1}^{\gamma-\alpha}=\tfrac12\varepsilon_{n+1}R_{n+1}^{\gamma-\alpha}=\tfrac12 c_02^{\,n+1}=c_02^{\,n}.
\end{equation}
Let $n\ge0$ and $R_n\le|x|\le R_{n+1}$. Using \eqref{eq_compconc}, \eqref{eq_identity} and \eqref{eq_An},
\[
|u(x)|\le|x|^{-\gamma}\left(A_n|x|^{\gamma-\alpha}+B_n\right)
\le|x|^{-\gamma}\left(A_nR_{n+1}^{\gamma-\alpha}+B_n\right)=2c_0\,2^{\,n}|x|^{-\gamma} .
\]

\smallskip
\emph{Step 3: bounding $2^n$ by a power of $\log|x|$.} Put $e_n=-\log\varepsilon_n$. From $\varepsilon_{n+1}=2K^{-1/(p-1)}\varepsilon_n^\kappa$ we get $e_{n+1}=\kappa e_n-c_1$, hence $e_{n+1}-\bar c=\kappa(e_n-\bar c)$ and
\[
e_n=\bar c+\kappa^{\,n}(e_0-\bar c)\ge \bar c+\kappa^{\,n},
\]
by the second condition in \eqref{eq_eps0}. Combining with \eqref{eq_identity} and $\log R_0\ge0$, $n\log2\ge0$,
\[
(\gamma-\alpha)\log R_n\ge e_n-e_0\ge\kappa^{\,n}+\bar c-e_0 ,
\]
so that $\kappa^{\,n}\le(\gamma-\alpha)\log R_n+e_0-\bar c$. Set $R_*=\max\{R_1,e\}$ and let $|x|\ge R_*$, so that $|x|\in[R_n,R_{n+1}]$ for some $n\ge1$ and $\log|x|\ge\log R_n\ge1$. Then
\[
\kappa^{\,n}\le(\gamma-\alpha)\log R_n+(e_0-\bar c)\le C_1\log|x| ,
\]
where we used $\log R_n\le\log|x|$, $\log|x|\ge1$ and $e_0-\bar c\ge1$, and where $C_1=\gamma-\alpha+e_0-\bar c$. Whence $n\le\frac{\log(C_1\log|x|)}{\log\kappa}$ and
\[
2^{\,n}\le\exp\left(\frac{\log 2}{\log\kappa}\log\big(C_1\log|x|\big)\right)=\big(C_1\log|x|\big)^{\theta} .
\]
Together with Step 2 this gives \eqref{eq_pharmonic} with $C=2c_0C_1^{\theta}$.
\end{proof}

\begin{corollary}\label{cor_gradboot}
Under the hypotheses of Proposition \ref{prop_boot} there are $C>0$ and $R_*>1$ with
\begin{equation}\label{eq_gradboot}
|\nabla u(x)|\le C|x|^{-\gamma-1}\left(\log|x|\right)^{\theta}\qquad\text{for } |x|\ge R_* .
\end{equation}
\end{corollary}

\begin{proof}
Let $r\ge 2R_*$ and put $M_r=\sup_{\frac12 r\le|y|\le 2r}|u(y)|$, so that $M_r\le Cr^{-\gamma}(\log r)^{\theta}$ by \eqref{eq_pharmonic}. Let $|x_0|=r$ and apply Lemma \ref{lem_grad} to $v(y)=u(x_0+\frac r4y)$ on $B_2$, which solves $-\Delta_p v=(r/4)^p|v|^{q-1}v$ weakly with $\|v\|_{L^\infty(B_2)}\le M_r$. As in the proof of Proposition \ref{prop_decay},
\[
\frac r4|\nabla u(x_0)|\le C\left(M_r+\big((r/4)^pM_r^{q}\big)^{\frac1{p-1}}\right)
=C\,M_r\left(1+\big((r/4)^pM_r^{q-p+1}\big)^{\frac1{p-1}}\right).
\]
Since $\gamma>\alpha$ we have $r^{\alpha}M_r\le Cr^{\alpha-\gamma}(\log r)^{\theta}\to0$, hence $r^pM_r^{q-p+1}=(r^{\alpha}M_r)^{q-p+1}\le1$ for $r$ large. Therefore $|\nabla u(x_0)|\le CM_r/r$, which is \eqref{eq_gradboot}.
\end{proof}

\begin{remark}
The logarithmic factor in \eqref{eq_pharmonic} is an artifact of the iteration, and its exponent can be made as small as one wishes. The factor $2$ in $\varepsilon_{n+1}=2A_n$ is the price paid for passing from the two term bound \eqref{eq_compconc} to a single term at the radius $R_{n+1}$. Fix $\tau>0$ and choose $R_{n+1}$ instead so that at that radius the second term of \eqref{eq_compconc} is $\tau$ times the first, that is
\[
\varepsilon_{n+1}=(1+\tau)A_n,\qquad
R_{n+1}=\Big(\frac{\varepsilon_n}{\tau A_n}\Big)^{\frac1{\gamma-\alpha}}R_n .
\]
The induction of Step 1 goes through verbatim once $\varepsilon_0$ is small enough, depending on $\tau$, and the bookkeeping of Step 2 now gives
\[
\varepsilon_nR_n^{\gamma-\alpha}=c_0\lambda^{\,n},\qquad \lambda=\frac{1+\tau}{\tau},
\]
in place of \eqref{eq_identity}, whence $|u(x)|\le c_0\lambda^{\,n+1}|x|^{-\gamma}$ when $R_n\le|x|\le R_{n+1}$ and, by Step 3 with $\log\lambda$ in place of $\log 2$,
\[
|u(x)|\le C_\tau\,|x|^{-\gamma}\big(\log|x|\big)^{\theta_\tau},
\qquad \theta_\tau=\frac{\log\frac{1+\tau}{\tau}}{\log\frac{q}{p-1}} .
\]
Since $\theta_\tau\to0$ as $\tau\to\infty$, the exponent $\theta$ in \eqref{eq_pharmonic} may be replaced by any positive number, at the cost of a larger constant. It cannot be replaced by $0$ this way, because $C_\tau$ grows with $\tau$, and we do not know whether $u(x)=O(|x|^{-\gamma})$ holds. The question is of no consequence for Theorem \ref{thm_main}, since all the inequalities of Lemma \ref{lem_energy} are strict.
\end{remark}

\section{Finite energy, the Pohozaev identity and the proof of Theorems \ref{thm_main} and \ref{thm_decay}}\label{sec_pohozaev}

\begin{lemma}\label{lem_energy}
Let $N>p\ge2$ and $\frac{N(p-1)}{N-p}<q$, and let $u\in C^1(\mathbb{R}^N)$ satisfy \eqref{eq_pharmonic} and \eqref{eq_gradboot}, with $\theta$ as in Proposition \ref{prop_boot}. Then
\[
\int_{\mathbb{R}^N}|\nabla u|^p\,dx<+\infty,\qquad
\int_{\mathbb{R}^N}|u|^{q+1}\,dx<+\infty,\qquad
\int_{\mathbb{R}^N}|u|^{p^*}\,dx<+\infty .
\]
\end{lemma}

\begin{proof}
Since $(p-1)\gamma=N-p$ we have
\[
N-(\gamma+1)p=N-p-\gamma p=(N-p)\left(1-\frac{p}{p-1}\right)=-\frac{N-p}{p-1}<0 .
\]
Also $q+1>\frac{N(p-1)}{N-p}=\frac N\gamma$, hence $\gamma(q+1)>N$, that is $N-\gamma(q+1)<0$. Write $\mu_1=(\gamma+1)p-N>0$ and $\mu_2=\gamma(q+1)-N>0$. By \eqref{eq_pharmonic} and \eqref{eq_gradboot},
\[
\int_{|x|\ge R_*}|\nabla u|^p\,dx\le C\int_{R_*}^{\infty}r^{-(\gamma+1)p}(\log r)^{\theta p}r^{N-1}\,dr
=C\int_{R_*}^{\infty}r^{-1-\mu_1}(\log r)^{\theta p}\,dr<+\infty,
\]
and similarly $\int_{|x|\ge R_*}|u|^{q+1}\,dx\le C\int_{R_*}^{\infty}r^{-1-\mu_2}(\log r)^{\theta(q+1)}\,dr<+\infty$. The same computation applies to $|u|^{p^*}$, because $\gamma p^*=\frac{Np}{p-1}>N$. On the ball $B_{R_*}$ all three integrands are continuous, hence integrable.
\end{proof}

The Pohozaev identity for the $p$-Laplacian, in the global form that finite energy makes available, is due to Damascelli, Farina, Sciunzi and Valdinoci, who give a self-contained proof for $C^{1}$ weak solutions.

\begin{lemma}[{\cite[Cor.~8.3 and Rem.~8.4]{DFSV2009}}]\label{lem_poho}
Let $u\in C^1(\mathbb{R}^N)$ be a weak solution of \eqref{eq_main} with $|\nabla u|\in L^p(\mathbb{R}^N)$ and $u\in L^{q+1}(\mathbb{R}^N)$. Then
\begin{equation}
\frac{N}{q+1}\int_{\mathbb{R}^N}|u|^{q+1}\,dx=\frac{N-p}{p}\int_{\mathbb{R}^N}|\nabla u|^p\,dx .
\end{equation}
\end{lemma}

\begin{proof}[Proof of Theorem \ref{thm_decay}]
Let $\varepsilon_*$ be the constant of Proposition \ref{prop_boot}. That proposition and Corollary \ref{cor_gradboot} give \eqref{eq_pharmonic} and \eqref{eq_gradboot}, and Lemma \ref{lem_energy} then gives
\[
u\in L^{q+1}(\mathbb{R}^N),\qquad |\nabla u|\in L^{p}(\mathbb{R}^N).
\]

Let $\eta_R\in C_c^\infty(B_{2R})$ with $0\le\eta_R\le1$, $\eta_R\equiv1$ on $B_R$ and $|\nabla\eta_R|\le C/R$. Testing \eqref{eq_main} with $u\eta_R$ gives
\[
\int_{\mathbb{R}^N}|\nabla u|^p\eta_R\,dx+\int_{\mathbb{R}^N}u\,|\nabla u|^{p-2}\nabla u\cdot\nabla\eta_R\,dx
=\int_{\mathbb{R}^N}|u|^{q+1}\eta_R\,dx .
\]
By \eqref{eq_pharmonic} and \eqref{eq_gradboot}, the middle term is bounded in absolute value by
\[
\frac{C}{R}\int_{B_{2R}\setminus B_R}|u|\,|\nabla u|^{p-1}\,dx
\le \frac{C}{R}\,R^{N}\,R^{-\gamma}R^{-(\gamma+1)(p-1)}(\log R)^{\theta p}
=C\,R^{\,N-p-\gamma p}(\log R)^{\theta p},
\]
and $N-p-\gamma p=-\frac{N-p}{p-1}<0$, so this tends to $0$. Letting $R\to\infty$ and using dominated convergence on the other two terms,
\begin{equation}\label{eq_id1}
\int_{\mathbb{R}^N}|\nabla u|^p\,dx=\int_{\mathbb{R}^N}|u|^{q+1}\,dx .
\end{equation}
The first and third bounds of Lemma \ref{lem_energy} give $u\in\mathcal D^{1,p}(\mathbb R^N)$, so \eqref{eq_id1} is the energy identity of the statement.

On the other hand, Lemma \ref{lem_poho} applies and gives
\begin{equation}\label{eq_id2}
\frac{N-p}{p}\int_{\mathbb{R}^N}|\nabla u|^p\,dx=\frac{N}{q+1}\int_{\mathbb{R}^N}|u|^{q+1}\,dx .
\end{equation}
Combining \eqref{eq_id1} and \eqref{eq_id2},
\[
\left(\frac{N-p}{p}-\frac{N}{q+1}\right)\int_{\mathbb{R}^N}|u|^{q+1}\,dx=0 .
\]
The factor $\frac{N-p}{p}-\frac{N}{q+1}$ vanishes only when $q+1=\frac{Np}{N-p}=p^*$. Hence, if $q\ne p^*-1$, then $\int_{\mathbb{R}^N}|u|^{q+1}=0$ and $u\equiv0$.
\end{proof}

\begin{proof}[Proof of Theorem \ref{thm_main}]
Let $u$ be as in the statement, with $p^*-1<q<q_c(N,p)$. Proposition \ref{prop_decay} gives $|x|^\alpha|u(x)|\to0$, so the hypothesis of Theorem \ref{thm_decay} holds for any $\varepsilon_*>0$, and $q>p^*-1>\frac{N(p-1)}{N-p}$. Since $q\ne p^*-1$, Theorem \ref{thm_decay} gives $u\equiv0$.
\end{proof}

\begin{proof}[Proof of Corollary \ref{cor_full}]
For $p=2$ a weak solution $u\in C^1(\mathbb{R}^N)$ is of class $C^2$ by Schauder theory, since $t\mapsto|t|^{q-1}t$ is $C^1$ for $q>1$, so the whole statement is that of Farina \cite{Farina2007}. For $p>2$ and $p-1<q<p^*-1$ it is the Liouville theorem of \cite{DFSV2009} for solutions stable outside a compact set. For $p\ge2$ and $p^*-1<q<q_c(N,p)$ it is Theorem \ref{thm_main}. The statement about finite Morse index solutions follows from Remark \ref{rem_fmi}.
\end{proof}

\section{The singular solution and the exponent \texorpdfstring{$q_c(N,p)$}{qc(N,p)}}\label{sec_sing}

Throughout this section $N>p\ge2$ and $q>p^*-1$, and $\alpha$, $\gamma$, $c_s$ and $u_{x_0}$ are as in \eqref{eq_exponents} and \eqref{eq_cs}. We write $u_s=u_0$ and
\begin{equation}\label{eq_beta}
\beta=(\alpha+1)(p-2),\qquad \partial_r\varphi=\frac{x}{|x|}\cdot\nabla\varphi .
\end{equation}
Two elementary inequalities will be used repeatedly. First, $q>p^*-1$ is equivalent to
\begin{equation}\label{eq_alphasmall}
\alpha<\frac{N-p}{p},
\end{equation}
because $\alpha<\frac{N-p}{p}$ means $q-p+1>\frac{p^2}{N-p}$, that is $q>\frac{(p-1)N+p}{N-p}=p^*-1$. Since $\frac{N-p}{p}<\frac{N-p}{p-1}=\gamma$, we have $0<\alpha<\gamma$, so $c_s$ is well defined and positive. Second,
\begin{equation}\label{eq_Nbeta}
N-2-\beta=N-p-\alpha(p-2)>0,
\end{equation}
because $N-2-(\alpha+1)(p-2)=N-2-(p-2)-\alpha(p-2)=N-p-\alpha(p-2)$ and, by \eqref{eq_alphasmall}, $\alpha(p-2)<\frac{(N-p)(p-2)}{p}<N-p$.

\subsection{A weighted Hardy inequality and an algebraic identity}

\begin{lemma}\label{lem_hardy}
Let $N\ge2$, $\sigma\ge0$ and $\beta\in\mathbb{R}$ with $N-2-\beta\ne0$. Then
\begin{equation}\label{eq_hardy}
\int_{\mathbb{R}^N}|x|^{-\beta}\Big[|\nabla\varphi|^2+\sigma(\partial_r\varphi)^2\Big]dx
\ \ge\ (1+\sigma)\left(\frac{N-2-\beta}{2}\right)^2\int_{\mathbb{R}^N}|x|^{-\beta-2}\varphi^2\,dx
\end{equation}
for every $\varphi\in C_c^1(\mathbb{R}^N\setminus\{0\})$.
\end{lemma}

\begin{proof}
In polar coordinates $|\nabla\varphi|^2=(\partial_r\varphi)^2+r^{-2}|\nabla_\theta\varphi|^2$, where $\nabla_\theta$ is the gradient on $\mathbb{S}^{N-1}$. Discarding the angular term,
\[
\int_{\mathbb{R}^N}|x|^{-\beta}\Big[|\nabla\varphi|^2+\sigma(\partial_r\varphi)^2\Big]dx
\ \ge\ (1+\sigma)\int_{\mathbb{S}^{N-1}}\int_0^{\infty}\big(\partial_r\varphi\big)^2 r^{\,\nu}\,dr\,d\theta,
\qquad \nu=N-1-\beta .
\]
Fix $\theta$ and write $f(r)=\varphi(r\theta)$, a function in $C_c^1(0,\infty)$. Since $\nu\ne1$,
\begin{align*}
\int_0^\infty f^2r^{\nu-2}\,dr
&=\frac{1}{\nu-1}\int_0^\infty f^2\big(r^{\nu-1}\big)'\,dr \\
&=-\frac{2}{\nu-1}\int_0^\infty ff'r^{\nu-1}\,dr \\
&\le\frac{2}{|\nu-1|}
\left(\int_0^\infty f^2r^{\nu-2}\,dr\right)^{1/2}
\left(\int_0^\infty f'^2r^{\nu}\,dr\right)^{1/2},
\end{align*}
so that $\int_0^\infty f'^2r^\nu\,dr\ge\big(\frac{\nu-1}{2}\big)^2\int_0^\infty f^2r^{\nu-2}dr$. Since $\frac{\nu-1}{2}=\frac{N-2-\beta}{2}$, integrating over $\mathbb{S}^{N-1}$ gives \eqref{eq_hardy}.
\end{proof}

\begin{lemma}\label{lem_algebra}
Let $N>p\ge2$ and $q>p^*-1$, and set $t=q-p+1$ and $M=N-p$. Then
\begin{equation}\label{eq_Hdef}
4(p-1)t^2\left[\left(\frac{N-2-\beta}{2}\right)^2-q\,\alpha(\gamma-\alpha)\right]=H(t),
\end{equation}
where
\begin{equation}\label{eq_H}
H(t)=M\big[(p-1)M-4p\big]\,t^2+2p^2(p-1)(2-M)\,t+p^4(p-1) .
\end{equation}
One has $H\big(\frac{p^2}{M}\big)=-\frac{4p^4}{M}<0$, and
\[
H(t)\ge0\iff (p-1)N>p(p+3)\ \text{ and }\ q\ge q_c(N,p).
\]
\end{lemma}

\begin{proof}
By \eqref{eq_Nbeta}, $N-2-\beta=M-\alpha(p-2)$, and $\alpha=p/t$, $\gamma=M/(p-1)$, $q=t+p-1$. Hence
\begin{align*}
4t^2\left(\frac{N-2-\beta}{2}\right)^2
&=\big(Mt-p(p-2)\big)^2,\\
4t^2q\alpha(\gamma-\alpha)
&=4t^2(t+p-1)\frac pt\left(\frac{M}{p-1}-\frac pt\right)\\
&=\frac{4p}{p-1}(t+p-1)\big(Mt-p(p-1)\big).
\end{align*}
Multiplying the difference by $p-1$ and expanding,
\[
(p-1)\big(Mt-p(p-2)\big)^2-4p(t+p-1)\big(Mt-p(p-1)\big)
=\big[(p-1)M^2-4pM\big]t^2+c_1t+c_0 ,
\]
with
\begin{align*}
c_1
&=-2(p-1)Mp(p-2)-4p(p-1)M+4p^2(p-1)\\
&=-2p^2(p-1)M+4p^2(p-1)
 =2p^2(p-1)(2-M).
\end{align*}
and $c_0=(p-1)p^2(p-2)^2+4p^2(p-1)^2=p^2(p-1)\big[(p-2)^2+4(p-1)\big]=p^4(p-1)$. This is \eqref{eq_Hdef} with $H$ as in \eqref{eq_H}.

Next, with $t=p^2/M$,
\begin{align*}
\frac{M^2}{p^4}H\!\left(\frac{p^2}{M}\right)
&=M\big[(p-1)M-4p\big]+2(p-1)(2-M)M+(p-1)M^2\\
&=\big[(p-1)-2(p-1)+(p-1)\big]M^2-4pM+4(p-1)M\\
&=-4M.
\end{align*}
so $H(p^2/M)=-4p^4/M<0$. By \eqref{eq_alphasmall} the range $q>p^*-1$ is exactly $t>p^2/M$.

Write $D=M[(p-1)M-4p]=(N-p)\big[(p-1)N-p(p+3)\big]$, using $(p-1)M-4p=(p-1)N-p(p-1)-4p=(p-1)N-p(p+3)$. If $D\le0$ then $H$ is concave, $H(0)=p^4(p-1)>0$ and $H(p^2/M)<0$, so $H$ has a zero in $(0,p^2/M)$ and, being concave with $H(0)>0$, is negative on $[p^2/M,\infty)$. If $D>0$ then $H$ is a convex parabola and $H(p^2/M)<0$, so $p^2/M$ lies strictly between its two zeros and $H(t)\ge0$ on $(p^2/M,\infty)$ if and only if $t\ge t_+$, the larger zero. The discriminant of $H$ divided by $4$ equals
\begin{align*}
p^4(p-1)^2(M-2)^2-Dp^4(p-1)
&=p^4(p-1)\Big[(p-1)(M-2)^2-M\big((p-1)M-4p\big)\Big]\\
&=4p^4(p-1)\big(M+p-1\big)\\
&=4p^4(p-1)(N-1).
\end{align*}
where we used $(p-1)(M^2-4M+4)-(p-1)M^2+4pM=4M+4(p-1)$ and $M+p-1=N-1$. Hence
\[
t_+=\frac{p^2\Big[(p-1)(M-2)+2\sqrt{(p-1)(N-1)}\Big]}{D}.
\]
It remains to check that $p-1+t_+$ is the number $q_c(N,p)$ of \eqref{eq_qc}. Put $B=(p-1)N-p$, so that $(p-1)M=B+p(2-p)$ and $(p-1)M-4p=B-p^2-2p$. Then
\begin{align*}
(p-1)D+p^2(p-1)(M-2)
&=\big(B+p(2-p)\big)\big(B-p^2-2p\big)\\
&\quad +p^2\big(B+p(2-p)\big)-2p^2(p-1)\\
&=\big(B+2p-p^2\big)\big(B-2p\big)-2p^2(p-1).
\end{align*}
which expands to $B^2-p^2B-4p^2+2p^3-2p^3+2p^2=B^2-p^2B-2p^2$. Therefore
\begin{align*}
p-1+t_+
&=\frac{(p-1)D+p^2(p-1)(M-2)+2p^2\sqrt{(p-1)(N-1)}}{D}\\
&=\frac{B^2-p^2B-2p^2+2p^2\sqrt{(p-1)(N-1)}}
{(N-p)\big[(p-1)N-p(p+3)\big]}.
\end{align*}
and substituting $-p^2B-2p^2=-p^2\big((p-1)N-p\big)-2p^2=-p^2(p-1)N+p^2(p-2)$ turns the numerator into the numerator of \eqref{eq_qc}. Since $q=p-1+t$, the equivalence follows.
\end{proof}

\subsection{Homogeneous profiles, and blow-down compactness at \texorpdfstring{$q=q_c$}{q=qc}}

\begin{proof}[Proof of Theorem \ref{thm_qc_hom}]
Since $U$ is homogeneous of degree $-\alpha$, write
\[
U(r\theta)=r^{-\alpha}\psi(\theta),\qquad r>0,\quad \theta\in\mathbb S^{N-1}.
\]
Set
\[
A(\theta)=\big(\alpha^2\psi^2+|\nabla_\theta\psi|^2\big)^{\frac{p-2}{2}},
\qquad m=(p-1)(\gamma-\alpha).
\]
The identities
\[
\nabla U=r^{-\alpha-1}\big(-\alpha\psi\,e_r+\nabla_\theta\psi\big),
\qquad
\alpha q=1+(\alpha+1)(p-1)
\]
show, by testing the weak equation with functions of the form $\eta(r)\zeta(\theta)$ and integrating the radial derivative by parts, that $\psi$ solves on $\mathbb S^{N-1}$
\begin{equation}\label{eq_sphere_qc}
-\operatorname{div}_\theta(A\nabla_\theta\psi)+\alpha m A\psi
=|\psi|^{q-1}\psi
\end{equation}
weakly. Testing \eqref{eq_sphere_qc} with $\psi$ gives
\begin{equation}\label{eq_sphere_energy}
I_1+\alpha m I_0=I_q,
\end{equation}
where
\[
I_1=\int_{\mathbb S^{N-1}}A|\nabla_\theta\psi|^2,\qquad
I_0=\int_{\mathbb S^{N-1}}A\psi^2,\qquad
I_q=\int_{\mathbb S^{N-1}}|\psi|^{q+1}.
\]
All these quantities are finite. Indeed, the homogeneity of $U$ and the assumptions on any fixed annulus imply
$\psi\in W^{1,p}(\mathbb S^{N-1})\cap L^{q+1}(\mathbb S^{N-1})$. In particular,
$A\in L^{p/(p-2)}(\mathbb S^{N-1})$ when $p>2$ (with the evident interpretation when $p=2$), and
$A\nabla_\theta\psi\in L^{p/(p-1)}(\mathbb S^{N-1})$. Smooth functions are dense in
$W^{1,p}(\mathbb S^{N-1})\cap L^{q+1}(\mathbb S^{N-1})$ for the sum norm. Hence one may first test with smooth approximations of $\psi$ and then pass to the limit by H\"older's inequality; this justifies both testing \eqref{eq_sphere_qc} with $\psi$ and the nonsmooth angular factor used in the stability test below.

We next note that the assumed stability outside $B_R$ upgrades, by homogeneity, to stability on the whole punctured space. Indeed, if $Q_U\ge0$ for test functions supported outside $B_R$ and $\varphi\in C_c^1(\mathbb R^N\setminus\{0\})$, then for $t$ sufficiently large the function $\varphi_t(x)=\varphi(x/t)$ is supported outside $B_R$. Since $\alpha(q-1)=2+(\alpha+1)(p-2)$, a change of variables gives
\[
Q_U(\varphi_t)=t^{N-2-(\alpha+1)(p-2)}Q_U(\varphi),
\]
so $Q_U(\varphi)\ge0$.

Put
\[
\beta=(\alpha+1)(p-2),\qquad
\mu=\frac{N-2-\beta}{2}.
\]
Choose $\eta\in C_c^1(\mathbb R)$ and use in the stability inequality the test function
\[
\varphi(r,\theta)=r^{-\mu}\eta(\log r)\psi(\theta).
\]
For $p\ge2$,
\[
|\nabla U|^{p-2}|\nabla\varphi|^2
 +(p-2)|\nabla U|^{p-4}(\nabla U\cdot\nabla\varphi)^2
\le (p-1)|\nabla U|^{p-2}|\nabla\varphi|^2.
\]
Using $\alpha(q-1)=\beta+2$ and $2\mu=N-2-\beta$, the radial powers cancel after the change of variables $s=\log r$, and stability yields
\[
q I_q\int_{\mathbb R}\eta^2\,ds
\le(p-1)\left[
I_0\int_{\mathbb R}(\eta'-\mu\eta)^2\,ds
+I_1\int_{\mathbb R}\eta^2\,ds
\right].
\]
Take $\eta_L\in C_c^1(\mathbb R)$ with $0\le\eta_L\le1$, $\eta_L\equiv1$ on $[1,L-1]$, $\operatorname{supp}\eta_L\subset[0,L]$ and $|\eta_L'|\le2$, so that $\int\eta_L^2\ge L-2$ and $\int\eta_L'^2\le8$. Since $\int\eta_L\eta_L'=0$, dividing by $\int\eta_L^2$ and letting $L\to\infty$ give
\begin{equation}\label{eq_sphere_stab}
q I_q\le(p-1)(\mu^2 I_0+I_1).
\end{equation}
Combining \eqref{eq_sphere_energy} and \eqref{eq_sphere_stab},
\begin{equation}\label{eq_sphere_cancel}
(q-p+1)I_1
\le\big[(p-1)\mu^2-q\alpha m\big]I_0.
\end{equation}
Since $m=(p-1)(\gamma-\alpha)$, the coefficient on the right of \eqref{eq_sphere_cancel} is
\[
(p-1)\mu^2-q\alpha m=(p-1)\big[\mu^2-q\alpha(\gamma-\alpha)\big]=\frac{H(t)}{4t^2},
\qquad t=q-p+1,
\]
by the identity \eqref{eq_Hdef} of Lemma \ref{lem_algebra}. That lemma also tells us the sign of $H(t)$ in the range $q>p^*-1$: it is negative for $q<q_c(N,p)$ and zero at $q=q_c(N,p)$.

Suppose first $q<q_c(N,p)$. Then the right hand side of \eqref{eq_sphere_cancel} is a negative multiple of $I_0$, while the left hand side is nonnegative. Hence $I_0=0$. Since $A\psi^2>0$ at every $\theta$ with $\psi(\theta)\ne0$, this forces $\psi=0$ almost everywhere, that is $U\equiv0$. This is (i).

Suppose now $q=q_c(N,p)$, so that the coefficient on the right of \eqref{eq_sphere_cancel} vanishes. Because $q-p+1>0$, we obtain $I_1=0$. For $p=2$ we have $A\equiv1$ and this already gives $\nabla_\theta\psi=0$ almost everywhere. For $p>2$ the pointwise inequality
\[
A|\nabla_\theta\psi|^2=\big(\alpha^2\psi^2+|\nabla_\theta\psi|^2\big)^{\frac{p-2}{2}}|\nabla_\theta\psi|^2\ \ge\ |\nabla_\theta\psi|^{p}
\]
turns $I_1=0$ into $\int_{\mathbb S^{N-1}}|\nabla_\theta\psi|^p=0$, so again $\nabla_\theta\psi=0$ almost everywhere. In either case $\psi$ is constant. Equation \eqref{eq_sphere_qc} then shows that either $\psi=0$ or
\[
|\psi|^{q-p+1}=\alpha^{p-1}(p-1)(\gamma-\alpha)=c_s^{q-p+1}.
\]
Thus $U\equiv0$ or $U=\pm c_s|x|^{-\alpha}$, which is (ii).
\end{proof}

The endpoint estimates also give compactness of the rescaled family. This does not by itself imply homogeneity, but it isolates precisely what is missing from a complete classification at $q=q_c$.

\begin{proposition}[Blow-down compactness at $q_c$]\label{prop_qc_blowdown}
Assume $N>p\ge2$, $(p-1)N>p(p+3)$ and $q=q_c(N,p)$. Let $u\in C^1(\mathbb R^N)$ be a weak solution of \eqref{eq_main} which is stable outside a compact set, and define
\[
u_R(x)=R^\alpha u(Rx).
\]
For every sequence $R_j\to\infty$ there are a subsequence, not relabelled, and
\[
U\in W^{1,p}_{\mathrm{loc}}(\mathbb R^N\setminus\{0\})\cap L^{q+1}_{\mathrm{loc}}(\mathbb R^N\setminus\{0\})
\]
such that
\[
u_{R_j}\to U\quad\text{strongly in }W^{1,p}_{\mathrm{loc}}(\mathbb R^N\setminus\{0\})
\quad\text{and in }L^{q+1}_{\mathrm{loc}}(\mathbb R^N\setminus\{0\}).
\]
The limit $U$ is a weak solution of \eqref{eq_main} in $\mathbb R^N\setminus\{0\}$ and is stable there. In particular, if such a blow-down limit is homogeneous of degree $-\alpha$, then Theorem \ref{thm_qc_hom} gives
\[
U\equiv0\qquad\text{or}\qquad U=\pm c_s|x|^{-\alpha}.
\]
\end{proposition}

\begin{proof}
At $q=q_c$ we have, in the notation of Remark \ref{rem_qcbig},
\[
\delta_0=\delta_1:=\frac{N(q-p+1)}p-q>1.
\]
Fix three annuli $\Omega\Subset\Omega'\Subset\Omega''\Subset\mathbb R^N\setminus\{0\}$. For any $\delta\in(1,\delta_0)$, Corollary \ref{cor_annuli}, after scaling and covering $\Omega''$ by finitely many annuli of fixed ratio, gives
\begin{equation}\label{eq_blow_Lqd}
\sup_j\int_{\Omega''}|u_{R_j}|^{q+\delta}\,dx<\infty.
\end{equation}
Indeed, the power of $R_j$ cancels because $\alpha(q+\delta)=p(q+\delta)/(q-p+1)$.

Let $\chi''\in C_c^1(\Omega'')$ satisfy $\chi''\equiv1$ on $\Omega'$. Testing the equation for $u_{R_j}$ with $u_{R_j}(\chi'')^p$ and using Young's inequality yields the local Caccioppoli bound
\[
\int_{\Omega'}|\nabla u_{R_j}|^p\,dx
\le C\int_{\Omega''}\big(|u_{R_j}|^{q+1}+|u_{R_j}|^p\big)\,dx\le C,
\]
where the last inequality follows from \eqref{eq_blow_Lqd} and from $q+1<q+\delta$. Hence, after passing to a subsequence,
$u_{R_j}\rightharpoonup U$ in $W^{1,p}(\Omega')$ and $u_{R_j}\to U$ strongly in $L^p(\Omega')$. Interpolating the latter convergence with the uniform $L^{q+\delta}$ bound gives
\begin{equation}\label{eq_blow_Lq1}
u_{R_j}\to U\qquad\text{strongly in }L^{q+1}(\Omega').
\end{equation}

We upgrade the gradient convergence. Let $\chi\in C_c^1(\Omega')$ satisfy $\chi\equiv1$ on $\Omega$, put $\mathcal A(\xi)=|\xi|^{p-2}\xi$ and $w_j=u_{R_j}-U$. Testing the equation of $u_{R_j}$ with $\chi^p w_j$ gives
\[
\int\chi^p\mathcal A(\nabla u_{R_j})\cdot\nabla w_j
=\int |u_{R_j}|^{q-1}u_{R_j}\chi^p w_j
-p\int\chi^{p-1}w_j\mathcal A(\nabla u_{R_j})\cdot\nabla\chi.
\]
The first term on the right tends to zero by \eqref{eq_blow_Lq1}, and the second tends to zero by the strong $L^p$ convergence of $w_j$ and the uniform $L^{p/(p-1)}$ bound of $\mathcal A(\nabla u_{R_j})$. Since $\nabla u_{R_j}\rightharpoonup\nabla U$ in $L^p$,
\[
\int\chi^p\mathcal A(\nabla U)\cdot(\nabla u_{R_j}-\nabla U)\to0.
\]
The uniform monotonicity inequality for $p\ge2$ therefore gives
\[
\int_{\Omega}|\nabla u_{R_j}-\nabla U|^p\,dx\to0.
\]
Consequently $u_{R_j}\to U$ strongly in $W^{1,p}(\Omega)$, and passing to the limit in the weak formulation shows that $U$ solves \eqref{eq_main} on $\Omega$.

So far the subsequence depends on $\Omega$. Applying what precedes to the exhaustion $\Omega_k=\{k^{-1}<|x|<k\}$, $k\ge2$, and extracting a diagonal subsequence, we obtain one subsequence along which $u_{R_j}$ converges strongly in $W^{1,p}(\Omega_k)$ and in $L^{q+1}(\Omega_k)$ for every $k$; the limits agree on the overlaps and define a single $U$ on $\mathbb R^N\setminus\{0\}$, which solves \eqref{eq_main} there.

Finally let $\varphi\in C_c^1(\mathbb R^N\setminus\{0\})$. A change of variables as in the proof of Theorem \ref{thm_qc_hom} gives $Q_{u_R}(\varphi)=R^{\,2+(\alpha+1)(p-2)-N}\,Q_u\big(\varphi(\cdot/R)\big)$, the factor being positive. For $j$ large the support of $\varphi(\cdot/R_j)$ lies outside the compact set where $u$ may be unstable, hence $Q_{u_{R_j}}(\varphi)\ge0$. Let $\Omega_k$ contain the support of $\varphi$ and let $p>2$. The strong convergence of $\nabla u_{R_j}$ in $L^p(\Omega_k)$ and Lemma \ref{lem_nem}(ii) give
$\mathcal M(\nabla u_{R_j})\to\mathcal M(\nabla U)$ in $L^{p/(p-2)}(\Omega_k)$, and the strong convergence of $u_{R_j}$ in $L^{q+1}(\Omega_k)$ with Lemma \ref{lem_nem}(i) gives
$|u_{R_j}|^{q-1}\to|U|^{q-1}$ in $L^{(q+1)/(q-1)}(\Omega_k)$. Since $Q_{u_{R_j}}(\varphi)=\int\mathcal M(\nabla u_{R_j})\nabla\varphi\cdot\nabla\varphi-q\int|u_{R_j}|^{q-1}\varphi^2$, it follows that
$Q_{u_{R_j}}(\varphi)\to Q_U(\varphi)$. For $p=2$ the matrix coefficient is constant and the same conclusion is immediate. Thus $Q_U(\varphi)\ge0$.
\end{proof}

\begin{proposition}[$\omega$-limit structure at $q_c$]\label{prop_qc_omega}
Under the hypotheses of Proposition \ref{prop_qc_blowdown}, equip
$\mathcal X=W^{1,p}_{\mathrm{loc}}(\mathbb R^N\setminus\{0\})$ with its usual Fr\'echet topology and set
\[
\Omega_\infty(u)=\left\{U\in\mathcal X:\ u_{R_j}\to U\text{ in }\mathcal X\text{ for some }R_j\to\infty\right\}.
\]
Then $\Omega_\infty(u)$ is a nonempty compact connected subset of $\mathcal X$, and it is invariant under the scaling action
\[
(\mathcal S_\lambda U)(x)=\lambda^\alpha U(\lambda x),\qquad \lambda>0.
\]
Finally, the following statements are equivalent:
\begin{enumerate}[label=\rm(\roman*)]
\item every element of $\Omega_\infty(u)$ is homogeneous of degree $-\alpha$;
\item $\Omega_\infty(u)$ consists of a single element;
\item the full family $u_R$ converges in $\mathcal X$ as $R\to\infty$.
\end{enumerate}
Whenever these conditions hold, the limit is necessarily
\[
0,\qquad c_s|x|^{-\alpha},\qquad\text{or}\qquad -c_s|x|^{-\alpha}.
\]
\end{proposition}

\begin{proof}
Write $v(t)=u_{e^t}$ for $t\ge0$. The map $t\mapsto v(t)$ is continuous from $[0,\infty)$ into $\mathcal X$. For $T\ge0$ let
\[
K_T=\overline{\{v(t):t\ge T\}}^{\,\mathcal X}.
\]
Proposition \ref{prop_qc_blowdown} implies that $K_T$ is compact. Indeed, from any sequence $v(t_j)$ with $t_j\ge T$, either $(t_j)$ has a bounded subsequence, in which case continuity gives convergence after a further subsequence, or $t_j\to\infty$ along a subsequence, in which case Proposition \ref{prop_qc_blowdown} applies. Since $[T,\infty)$ is connected, its continuous image is connected, and therefore so is its closure $K_T$. The sets $K_T$ are nested, nonempty and compact, so their intersection is nonempty and compact, and it is connected because a nested intersection of compact connected subsets of a Hausdorff space is connected. Moreover,
\[
\Omega_\infty(u)=\bigcap_{T\ge0}K_T.
\]
The inclusion from left to right is immediate. Conversely, since the Fr\'echet space $\mathcal X$ is metrizable, if $U$ belongs to every $K_T$, then for each integer $n$ one can choose $t_n\ge n$ with $d_{\mathcal X}(v(t_n),U)<1/n$; hence $t_n\to\infty$ and $v(t_n)\to U$, so $U\in\Omega_\infty(u)$. Therefore $\Omega_\infty(u)$ is nonempty, compact and connected.

If $U\in\Omega_\infty(u)$ and $u_{R_j}\to U$, then for every $\lambda>0$,
\[
u_{\lambda R_j}=\mathcal S_\lambda u_{R_j}\longrightarrow \mathcal S_\lambda U
\quad\text{in }\mathcal X,
\]
so $\mathcal S_\lambda U\in\Omega_\infty(u)$. Thus the set is scaling-invariant.

If (iii) holds then every subsequential limit of $u_R$ equals the full limit, so $\Omega_\infty(u)$ is a singleton and (ii) holds. Conversely, if $\Omega_\infty(u)=\{U\}$ but the full family did not converge to $U$, there would be an $\mathcal X$-neighborhood $\mathcal U$ of $U$ and a sequence $R_j\to\infty$ with $u_{R_j}\notin\mathcal U$. Proposition \ref{prop_qc_blowdown} would then yield a subsequence converging in $\mathcal X$ to an element of $\Omega_\infty(u)$ distinct from $U$, a contradiction. Hence (ii) implies (iii).

If (iii) holds with limit $U$, then for every $\lambda>0$,
\[
\mathcal S_\lambda U
=\lim_{R\to\infty}\mathcal S_\lambda u_R
=\lim_{R\to\infty}u_{\lambda R}=U,
\]
so $U$ is homogeneous of degree $-\alpha$. Thus (iii) implies (i).

Finally assume (i). Every element of $\Omega_\infty(u)$ is a stable weak solution by Proposition \ref{prop_qc_blowdown}, so Theorem \ref{thm_qc_hom} gives
\[
\Omega_\infty(u)\subset\{0,\ c_s|x|^{-\alpha},\ -c_s|x|^{-\alpha}\}.
\]
The set on the right is a finite, hence discrete, subset of $\mathcal X$, whereas $\Omega_\infty(u)$ is connected. Hence $\Omega_\infty(u)$ is a singleton, proving (i) implies (ii) and the final alternative.
\end{proof}

\begin{remark}[What remains at $q=q_c$]
Proposition \ref{prop_qc_blowdown} shows that the rescaled orbit is precompact on every annulus, while Theorem \ref{thm_qc_hom} classifies all homogeneous stable limits without a sign assumption. Consequently, a complete asymptotic classification at the Joseph--Lundgren endpoint would follow once one proves that every blow-down limit is homogeneous. Proposition \ref{prop_qc_omega} shows that this is in fact equivalent to uniqueness of the blow-down: either property forces the full rescaled family to converge to one of $0$ and $\pm c_s|x|^{-\alpha}$. This is exactly the point at which the lack of a monotonicity formula for $p\ne2$ remains relevant. We do not claim such a homogeneity theorem here.
\end{remark}

\subsection{Proof of Theorem \ref{thm_sing} and of Proposition \ref{prop_radial}}

\begin{proof}[Proof of Theorem \ref{thm_sing}]
By translation we may take $x_0=0$ and write $u_s(x)=c_s|x|^{-\alpha}$.

\emph{Step 1: $u_s$ is a weak solution.} Lemma \ref{lem_super} with $B=0$ and $A=c_s$ gives, in $\mathbb{R}^N\setminus\{0\}$,
\[
-\Delta_p u_s=\alpha^{p-1}(p-1)(\gamma-\alpha)\,c_s^{p-1}|x|^{-\alpha q}=c_s^{q}|x|^{-\alpha q}=u_s^{q},
\]
by the definition \eqref{eq_cs} of $c_s$. By \eqref{eq_alphasmall}, $(\alpha+1)p<N$, so $\int_{B_1}|\nabla u_s|^p\,dx=C\int_0^1r^{N-1-(\alpha+1)p}dr<\infty$ and $u_s\in W^{1,p}_{\mathrm{loc}}(\mathbb{R}^N)$. Also $\alpha q=1+(\alpha+1)(p-1)=(\alpha+1)p-\alpha<N$, so $u_s^q\in L^1_{\mathrm{loc}}(\mathbb{R}^N)$. Let $\varphi\in C_c^1(\mathbb{R}^N)$ and let $\zeta_\varepsilon\in C^\infty(\mathbb{R}^N)$ satisfy $0\le\zeta_\varepsilon\le1$, $\zeta_\varepsilon\equiv0$ on $B_\varepsilon$, $\zeta_\varepsilon\equiv1$ outside $B_{2\varepsilon}$ and $|\nabla\zeta_\varepsilon|\le2/\varepsilon$. Testing the equation on $\mathbb{R}^N\setminus\{0\}$ with $\varphi\zeta_\varepsilon$,
\[
\int|\nabla u_s|^{p-2}\nabla u_s\cdot\nabla\varphi\,\zeta_\varepsilon
+\int|\nabla u_s|^{p-2}\varphi\,\nabla u_s\cdot\nabla\zeta_\varepsilon
=\int u_s^q\varphi\,\zeta_\varepsilon .
\]
The middle term is bounded in absolute value by
\[
\frac{2\|\varphi\|_\infty}{\varepsilon}\int_{B_{2\varepsilon}}(\alpha c_s)^{p-1}|x|^{-(\alpha+1)(p-1)}dx
\le\frac{C}{\varepsilon}\,\varepsilon^{\,N-(\alpha+1)(p-1)}=C\varepsilon^{\,(p-1)(\gamma-\alpha)}\xrightarrow[\varepsilon\to0]{}0,
\]
because $N-1-(\alpha+1)(p-1)=(p-1)(\gamma-\alpha)>0$. The two remaining terms converge by dominated convergence, so $u_s$ is a weak solution of \eqref{eq_main} in $\mathbb{R}^N$.

\emph{Step 2: the quadratic form.} On $\mathbb{R}^N\setminus\{0\}$ we have $\nabla u_s=-\alpha c_s|x|^{-\alpha-1}\frac{x}{|x|}$, hence $|\nabla u_s|=\alpha c_s|x|^{-\alpha-1}$ and $\nabla u_s\cdot\nabla\varphi=-\alpha c_s|x|^{-\alpha-1}\partial_r\varphi$. Therefore
\[
|\nabla u_s|^{p-4}\big(\nabla u_s\cdot\nabla\varphi\big)^2=(\alpha c_s)^{p-2}|x|^{-\beta}\big(\partial_r\varphi\big)^2 ,
\]
with $\beta$ as in \eqref{eq_beta}, and, since $\alpha(q-1)=\alpha q-\alpha=1+(\alpha+1)(p-1)-\alpha=\beta+2$,
\begin{equation}\label{eq_Qsing}
Q_{u_s}(\varphi)=(\alpha c_s)^{p-2}\int_{\mathbb{R}^N}|x|^{-\beta}\Big[|\nabla\varphi|^2+(p-2)\big(\partial_r\varphi\big)^2\Big]dx
-q\,c_s^{q-1}\int_{\mathbb{R}^N}|x|^{-\beta-2}\varphi^2\,dx
\end{equation}
for every $\varphi\in C_c^1(\mathbb{R}^N\setminus\{0\})$. By \eqref{eq_cs},
\begin{equation}\label{eq_Lambda}
\frac{q\,c_s^{q-1}}{(\alpha c_s)^{p-2}}=\frac{q\,c_s^{q-p+1}}{\alpha^{p-2}}=q\,\alpha(p-1)(\gamma-\alpha) .
\end{equation}

\emph{Step 3: (iii) implies (i).} Assume $q\ge q_c(N,p)$. By Lemma \ref{lem_algebra} and \eqref{eq_Hdef},
\[
\left(\frac{N-2-\beta}{2}\right)^2\ge q\,\alpha(\gamma-\alpha),
\]
so, multiplying by $p-1$ and using \eqref{eq_Lambda}, Lemma \ref{lem_hardy} with $\sigma=p-2$ gives, for every $\varphi\in C_c^1(\mathbb{R}^N\setminus\{0\})$,
\[
\int|x|^{-\beta}\Big[|\nabla\varphi|^2+(p-2)(\partial_r\varphi)^2\Big]dx
\ge(p-1)\left(\frac{N-2-\beta}{2}\right)^2\int|x|^{-\beta-2}\varphi^2dx
\ge\frac{q\,c_s^{q-1}}{(\alpha c_s)^{p-2}}\int|x|^{-\beta-2}\varphi^2dx .
\]
With \eqref{eq_Qsing} this is $Q_{u_s}(\varphi)\ge0$, which is (i). Note that \eqref{eq_Nbeta} guarantees $N-2-\beta\ne0$, so Lemma \ref{lem_hardy} applies.

\emph{Step 4: (i) implies (ii).} Immediate, with $K=\overline{B_1}$.

\emph{Step 5: (ii) implies (iii).} The function $u_s$ is $(-\alpha)$-homogeneous and not identically zero. If (ii) holds and $q<q_c(N,p)$, then part (i) of Theorem \ref{thm_qc_hom} gives $u_s\equiv0$, a contradiction. Hence $q\ge q_c(N,p)$.
\end{proof}

\begin{proof}[Proof of Proposition \ref{prop_radial}]
\emph{Step 1: the radial identity.} Write $u=u(r)$, $r=|x|$, and set
\[
G(r)=r^{N-1}|u'(r)|^{p-2}u'(r).
\]
Testing \eqref{eq_main} with radial test functions gives
\[
G'=-r^{N-1}u^q
\]
in the sense of distributions on $(0,\infty)$. The right hand side is continuous, so $G\in C^1(0,\infty)$. Since
\[
|G(r)|\le r^{N-1}\sup_{B_1}|\nabla u|^{p-1}\longrightarrow0
\qquad\text{as }r\to0,
\]
we obtain $G(r)=-\int_0^r s^{N-1}u(s)^q\,ds<0$ for $r>0$. Hence $u'<0$ on $(0,\infty)$, so $u$ is strictly decreasing and $|u'|^{p-2}u'=-(-u')^{p-1}$. Therefore
\begin{equation}\label{eq_radialint}
r^{N-1}\big(-u'(r)\big)^{p-1}=\int_0^r s^{N-1}u(s)^q\,ds\qquad(r>0).
\end{equation}
In the notation of Lemma \ref{lem_super},
\[
m=N-1-(\alpha+1)(p-1)=(p-1)(\gamma-\alpha)>0.
\]
Since $\alpha q=1+(\alpha+1)(p-1)$, this gives $\alpha q+m=N$.

\emph{Step 2: a planar autonomous system.} Put $t=\log r$ and
\[
W(t)=r^{\alpha}u(r),\qquad F(t)=\Big(r^{\,\alpha+1}\big(-u'(r)\big)\Big)^{p-1},\qquad r=e^{t}.
\]
Both are positive, by Step 1. Since $\dot W=\alpha W+r^{\alpha+1}u'$, we have $r^{\alpha+1}(-u')=\alpha W-\dot W$, which is therefore positive, and $F=(\alpha W-\dot W)^{p-1}$, that is
\[
\dot W=\alpha W-F^{\frac1{p-1}} .
\]
Substituting $s=e^{\tau}$ and $u(s)=e^{-\alpha\tau}W(\tau)$ in \eqref{eq_radialint}, and using $N-1-\alpha q=m-1$, gives
\[
e^{\,mt}F(t)=r^{N-1}(-u')^{p-1}=\int_0^{r}s^{\,m-1}\big(s^{\alpha}u(s)\big)^{q}ds
=\int_{-\infty}^{t}e^{\,m\tau}W(\tau)^{q}\,d\tau ,
\]
the integral converging at $-\infty$ because $m>0$ and $W(\tau)\to0$ as $\tau\to-\infty$. Hence
\[
F(t)=\int_{-\infty}^{t}e^{\,m(\tau-t)}W(\tau)^{q}\,d\tau,
\qquad\text{so}\qquad
\dot F=W^{q}-mF .
\]
Thus $(W,F)$ is a trajectory of the autonomous system
\begin{equation}\label{eq_phase}
\dot W=\alpha W-F^{\frac1{p-1}},\qquad \dot F=W^{q}-mF
\end{equation}
in the open quadrant $\mathcal Q=\{W>0,\ F>0\}$, where the vector field is of class $C^1$.

\emph{Step 3: the orbit is relatively compact in $\mathcal Q$.} Put $\ell=\liminf_{t\to\infty}W(t)$ and $L=\limsup_{t\to\infty}W(t)$, so that $0<\ell\le L<\infty$ by \eqref{eq_slow}. Let $\epsilon>0$ and pick $T$ with $\ell^{q}-\epsilon\le W(\tau)^{q}\le L^{q}+\epsilon$ for $\tau\ge T$. Splitting the integral for $F$ at $T$, using $\int_{-\infty}^{t}e^{m(\tau-t)}d\tau=\frac1m$ and the boundedness of $W$ on $(-\infty,T]$, the part over $(-\infty,T]$ is $O(e^{m(T-t)})$ and we obtain
\[
\frac{\ell^{q}-\epsilon}{m}\le\liminf_{t\to\infty}F(t)\le\limsup_{t\to\infty}F(t)\le\frac{L^{q}+\epsilon}{m} .
\]
As $\epsilon$ is arbitrary, $\liminf F\ge\ell^q/m>0$ and $\limsup F\le L^q/m<\infty$. Hence the forward orbit $\{(W(t),F(t)):t\ge T_0\}$ has compact closure in $\mathcal Q$ for $T_0$ large.

\emph{Step 4: the only equilibrium, and no cycles.} The equilibria of \eqref{eq_phase} in $\mathcal Q$ satisfy $F=W^{q}/m$ and $(\alpha W)^{p-1}=F$, hence $W^{q-p+1}=m\alpha^{p-1}=\alpha^{p-1}(p-1)(\gamma-\alpha)=c_s^{q-p+1}$. So $P=\big(c_s,(\alpha c_s)^{p-1}\big)$ is the unique equilibrium in $\mathcal Q$. The divergence of the vector field in \eqref{eq_phase} is the constant
\[
\alpha-m=\alpha-(p-1)(\gamma-\alpha)=p\alpha-(p-1)\gamma=p\alpha-(N-p)<0,
\]
the inequality being \eqref{eq_alphasmall}, that is $q>p^*-1$. Let $\Gamma$ be any periodic orbit of \eqref{eq_phase} and let $D$ be the region it encloses, so $D\subset\mathcal Q$ because $\mathcal Q$ is simply connected. The vector field is tangent to $\Gamma$, and Green's theorem gives $\int_D(\alpha-m)\,dWdF=0$, which is impossible since $|D|>0$. The same conclusion excludes a homoclinic loop at $P$: excise a disk $B_\rho(P)$, apply Green's theorem to the region enclosed by the loop with $B_\rho(P)$ removed, and let $\rho\downarrow0$; the flux across $\partial B_\rho(P)$ tends to zero because the vector field is $C^1$ and vanishes at $P$. Hence \eqref{eq_phase} has neither a periodic orbit nor a homoclinic loop in $\mathcal Q$.

\emph{Step 5: convergence.} By Step 3 the $\omega$-limit set of the forward orbit is a nonempty compact connected invariant subset of $\mathcal Q$. By the Poincar\'e--Bendixson theorem it is a periodic orbit, an equilibrium, or a union of equilibria and orbits connecting them. The first is excluded by Step 4. In the third case, $P$ being the only equilibrium in $\mathcal Q$, the set would contain a homoclinic loop at $P$, which Step 4 also excludes. Hence the $\omega$-limit set is an equilibrium, necessarily $P$. Therefore $W(t)\to c_s$ and $F(t)\to(\alpha c_s)^{p-1}$, which is exactly
\[
\lim_{r\to\infty}r^{\alpha}u(r)=c_s,\qquad \lim_{r\to\infty}r^{\alpha+1}\big(-u'(r)\big)=\alpha c_s ,
\]
and \eqref{eq_slowexact} follows because $|\nabla u|=-u'$.

\emph{Step 6: stability when $q>q_c(N,p)$.} Here Lemma \ref{lem_algebra} gives $H(q-p+1)>0$, that is
\[
(p-1)\left(\frac{N-2-\beta}{2}\right)^2>q\,\alpha(p-1)(\gamma-\alpha)=\frac{q\,c_s^{q-1}}{(\alpha c_s)^{p-2}} ,
\]
the last equality being \eqref{eq_Lambda}. Choose $\varepsilon\in(0,1)$ so small that
\[
(1-\varepsilon)^{p-2}(p-1)\left(\frac{N-2-\beta}{2}\right)^2\ \ge\ (1+\varepsilon)^{q-1}\,\frac{q\,c_s^{q-1}}{(\alpha c_s)^{p-2}} ,
\]
which is possible because the two sides are continuous in $\varepsilon$ and the inequality is strict at $\varepsilon=0$. By \eqref{eq_slowexact} there is $R_\varepsilon>0$ with
\[
u(x)\le(1+\varepsilon)c_s|x|^{-\alpha}
\qquad\text{and}\qquad
|\nabla u(x)|\ge(1-\varepsilon)\alpha c_s|x|^{-\alpha-1}
\qquad\text{for } |x|\ge R_\varepsilon .
\]
Let $\varphi\in C_c^1(\{|x|>R_\varepsilon\})$. Since $u$ is radial, $\nabla u=u'(|x|)\frac{x}{|x|}$ and therefore $|\nabla u|^{p-4}(\nabla u\cdot\nabla\varphi)^2=|\nabla u|^{p-2}(\partial_r\varphi)^2$, so
\[
Q_u(\varphi)=\int|\nabla u|^{p-2}\Big[|\nabla\varphi|^2+(p-2)(\partial_r\varphi)^2\Big]dx-q\int u^{q-1}\varphi^2dx .
\]
The bracket is nonnegative because $p\ge2$, so the first integral is at least
$\big((1-\varepsilon)\alpha c_s\big)^{p-2}\int|x|^{-\beta}\big[|\nabla\varphi|^2+(p-2)(\partial_r\varphi)^2\big]dx$, while the second is at most $q\big((1+\varepsilon)c_s\big)^{q-1}\int|x|^{-\beta-2}\varphi^2dx$, using $\alpha(q-1)=\beta+2$. Lemma \ref{lem_hardy} with $\sigma=p-2$ and the choice of $\varepsilon$ now give $Q_u(\varphi)\ge0$. Hence $u$ is stable outside $\overline{B_{R_\varepsilon}}$.
\end{proof}

\begin{remark}\label{rem_psmall}
Three steps use $p\ge2$. In Lemma \ref{lem_stab} the bound $|\nabla u|^{p-4}(\nabla u\cdot\nabla\varphi)^2\le|\nabla u|^{p-2}|\nabla\varphi|^2$ is applied with the factor $p-2\ge0$, and for $1<p<2$ the corresponding constant in front of the quadratic form is $1$ rather than $p-1$, which changes $\delta_0$ and hence the threshold. In Lemma \ref{lem_super} the factor $(1+t)^{p-2}$ is at least $1$ only when $p\ge2$, and for $1<p<2$ it is smaller than $1$, so the comparison function has to be corrected by an amount depending on the ratio $B/A$ at the inner radius. In Lemma \ref{lem_hardy} the parameter $\sigma=p-2$ has to be nonnegative, and for $1<p<2$ the constant in \eqref{eq_hardy} becomes $(p-1)\big(\frac{N-2-\beta}{2}\big)^2$ only if one first discards the angular term, which is wasteful. The first two points seem to be technical rather than structural. The third is not, since it would change the threshold in Theorem \ref{thm_sing} as well.
\end{remark}

\section{The Sobolev endpoint and the Morse index of the extremals}\label{sec_index}

\begin{proof}[Proof of Theorem \ref{thm_sob_endpoint}]
Let $q=p^*-1$. Proposition \ref{prop_decay} applies at the endpoint and gives $|x|^\alpha|u(x)|\to0$, and $p^*-1>\frac{N(p-1)}{N-p}$, so Theorem \ref{thm_decay} applies and yields \eqref{eq_critical_energy}. Its last assertion is vacuous here, which is why the Sobolev endpoint survives.

Assume now that $u$ has finite Morse index. Since $u\in\mathcal D^{1,p}(\mathbb R^N)$ by \eqref{eq_critical_energy}, \cite[Thm.~1.4]{FMW2019} gives $n(u)\le i(u)$ when $p>2$. We record the corresponding elementary argument for $p=2$. Let $\omega_1,\ldots,\omega_k$ be distinct nodal regions of $u$. By the nodal-domain approximation of \cite[Lem.~3.1]{FMW2019}, for each $j$ there are functions $v_{j,n}\in C_c^1(\omega_j)$ converging in $\mathcal D^{1,2}(\mathbb R^N)$ to the restriction of $u$ to $\omega_j$, extended by zero. Testing the equation on the nodal region gives
\[
\int_{\omega_j}|\nabla u|^2=\int_{\omega_j}|u|^{2^*},
\]
and therefore
\[
Q_u(v_{j,n})\longrightarrow -(2^*-2)\int_{\omega_j}|u|^{2^*}<0.
\]
For $n$ large the $k$ functions $v_{j,n}$ have pairwise disjoint supports, so the quadratic form is negative definite on their span. Thus $k\le i(u)$. Since $k$ was arbitrary, the total number $n(u)$ of nodal regions is finite and satisfies $n(u)\le i(u)$ also for $p=2$.

Suppose $i(u)\le1$. Then $n(u)\le1$, so $u$ does not change sign. It remains to treat a solution of one sign, with or without the Morse index hypothesis. Either $u\equiv0$, or the strong maximum principle makes $u$ positive after a change of sign, and the classification of positive $\mathcal D^{1,p}(\mathbb R^N)$ solutions of the critical equation gives $u=u_{\lambda,x_0}$ for some $\lambda>0$ and $x_0\in\mathbb R^N$. That classification is \cite{Sciunzi2016} for $p>2$ and \cite{CGS1989} for $p=2$.
\end{proof}

Throughout the remainder of this section $2<p<N$ and $q=p^*-1$. The quadratic form of Definition \ref{def_stable} satisfies $Q_{-u}=Q_u$, because it involves $u$ and $\nabla u$ only through $|u|$, $|\nabla u|$ and $(\nabla u\cdot\nabla\varphi)^2$, so the Morse index does not change under a change of sign. It is therefore enough to treat the positive member of the family \eqref{eq_bubble}, and for a fixed choice of $\lambda>0$ and $x_0\in\mathbb{R}^N$ we let $U$ denote it. We work in
\[
X:=\mathcal{D}^{1,p}(\mathbb{R}^N),\qquad \|u\|_X=\|\nabla u\|_{L^p},
\]
and with
\[
J(u)=\frac1p\int_{\mathbb{R}^N}|\nabla u|^p\,dx-\frac{1}{p^*}\int_{\mathbb{R}^N}|u|^{p^*}dx,
\qquad
\Phi(u)=\int_{\mathbb{R}^N}|\nabla u|^p\,dx-\int_{\mathbb{R}^N}|u|^{p^*}dx ,
\]
so that $\Phi(u)=J'(u)[u]$ and $\mathcal{N}=\{u\in X\setminus\{0\}:\Phi(u)=0\}$ is the Nehari manifold. Recall that $U\in X$, that $U$ is a critical point of $J$, and that $U$ attains the best constant
\[
S=\inf_{u\in X\setminus\{0\}}\frac{\|\nabla u\|_{L^p}^p}{\|u\|_{L^{p^*}}^p}
\]
in the Sobolev inequality, which is attained by the functions $u_{\lambda,x_0}$ \cite{Talenti1976}.

\begin{lemma}\label{lem_C2}
For $2<p<N$ the functionals $J$ and $\Phi$ are of class $C^2$ on $X$, and for every $u,v\in X$
\begin{equation}\label{eq_Jsecond}
J''(u)[v,v]=\int_{\mathbb{R}^N}|\nabla u|^{p-2}|\nabla v|^2dx+(p-2)\int_{\mathbb{R}^N}|\nabla u|^{p-4}(\nabla u\cdot\nabla v)^2dx-(p^*-1)\int_{\mathbb{R}^N}|u|^{p^*-2}v^2dx .
\end{equation}
In particular $J''(U)[v,v]=Q_U(v)$ for every $v\in X$, where $Q_U$ is the quadratic form of Section \ref{sec_intro} with $q=p^*-1$.
\end{lemma}

\begin{proof}
Let $\mathcal M$ be the matrix \eqref{eq_Mdef}. We argue in three steps, for the functional $u\mapsto\int|\nabla u|^p$; the functional $u\mapsto\int|u|^{p^*}$ is treated in the same way, with $p^*>p>2$ in place of $p$.

\emph{Step 1: the bilinear form.} For $u\in X$ put
\[
B_u[v,w]=\int_{\mathbb{R}^N} \mathcal M(\nabla u)\nabla v\cdot\nabla w\,dx,\qquad v,w\in X .
\]
By \eqref{eq_Mbound} and H\"older with exponents $\frac{p}{p-2}$, $p$, $p$,
\[
\big|B_u[v,w]\big|\le(p-1)\,\|\nabla u\|_{L^p}^{p-2}\,\|\nabla v\|_{L^p}\,\|\nabla w\|_{L^p},
\]
so $B_u$ is a bounded bilinear form on $X$.

\emph{Step 2: $B$ is the Gateaux differential of $I'$, in the norm of $X^*$, and it is continuous in $u$.} Write $I(u)=\frac1p\int|\nabla u|^p$, so that $I'(u)[v]=\int|\nabla u|^{p-2}\nabla u\cdot\nabla v\,dx$. By Lemma \ref{lem_nem}(ii) the map $\eta\mapsto|\eta|^{p-2}\eta$ is of class $C^1$ on all of $\mathbb{R}^N$ with differential $\mathcal M$, so for $u,w\in X$ and $0<|t|\le1$ the fundamental theorem of calculus gives, pointwise,
\[
\frac1t\Big(|\nabla u+t\nabla w|^{p-2}(\nabla u+t\nabla w)-|\nabla u|^{p-2}\nabla u\Big)
=\int_0^1 \mathcal M(\nabla u+\tau t\nabla w)\,\nabla w\,d\tau .
\]
Let $\mathcal{A}(u):=\mathcal M(\nabla u)$. Pairing the last identity with $\nabla v$, integrating, and applying H\"older as in Step 1 together with Minkowski's integral inequality in the $\tau$ variable,
\[
\Big\|\tfrac1t\big(I'(u+tw)-I'(u)\big)-B_u[w,\cdot\,]\Big\|_{X^*}
\ \le\ \Big(\int_0^1\big\|\mathcal{A}(u+\tau tw)-\mathcal{A}(u)\big\|_{L^{\frac{p}{p-2}}}d\tau\Big)\,\|\nabla w\|_{L^p} .
\]
By Lemma \ref{lem_nem}(ii), $\mathcal A$ is continuous from $X$ into $L^{\frac{p}{p-2}}(\mathbb{R}^N)$. The integrand in the bracket therefore tends to $0$ for each $\tau$ as $t\to0$, and is bounded by $2(p-1)\big\||\nabla u|+|\nabla w|\big\|_{L^p}^{p-2}$ uniformly in $\tau$ and $|t|\le1$, so the bracket tends to $0$ by dominated convergence. Hence $I'$ is Gateaux differentiable at $u$ with differential $B_u$, and $u\mapsto B_u$ is continuous from $X$ into the space of bounded bilinear forms on $X$, again by the continuity of $\mathcal A$ and the H\"older bound of Step 1.

\emph{Step 3: conclusion.} A Gateaux differentiable map between Banach spaces whose Gateaux differential is continuous is Fr\'echet differentiable, with the same differential. Applying this to $I':X\to X^*$, Steps 1 and 2 give that $I$ is of class $C^2$ on $X$ with $I''(u)=B_u$, hence that $u\mapsto\int|\nabla u|^p=pI(u)$ is of class $C^2$ with second derivative $p\,B_u$. The corresponding statement for $u\mapsto\int|u|^{p^*}$ gives the second derivative $p^*(p^*-1)\int|u|^{p^*-2}vw$, and \eqref{eq_Jsecond} follows, as does the corresponding formula for $\Phi''$. The identification with $Q_U$ is a comparison of the two expressions, using $q-1=p^*-2$.
\end{proof}

\begin{lemma}\label{lem_nehari}
$U$ minimizes $J$ on $\mathcal{N}$, and $\Phi'(U)\ne0$.
\end{lemma}

\begin{proof}
Since $U\in X$ and $U$ solves the equation weakly in $X$, testing with $U$ itself is admissible and gives $\int|\nabla U|^p=\int U^{p^*}$, that is $\Phi(U)=0$, so $U\in\mathcal N$. For $u\in\mathcal N$ put $Y=\|\nabla u\|_{L^p}^p=\|u\|^{p^*}_{L^{p^*}}>0$. The Sobolev inequality gives $Y\ge S\|u\|_{L^{p^*}}^p=S\,Y^{p/p^*}$, hence
\[
Y^{1-\frac{p}{p^*}}\ge S,\qquad\text{that is}\qquad Y\ge S^{\frac{p^*}{p^*-p}}=S^{\frac Np},
\]
and $Y=S^{N/p}$ for $u=U$, because $U$ attains $S$ \cite{Talenti1976} and lies on $\mathcal N$. Since $J(u)=\big(\frac1p-\frac1{p^*}\big)Y$ on $\mathcal{N}$, the function $U$ is a global minimizer of $J$ on $\mathcal N$.

Finally
\[
\Phi'(U)[U]=p\int|\nabla U|^p-p^*\int U^{p^*}=(p-p^*)\int U^{p^*}<0 ,
\]
where we used $\int|\nabla U|^p=\int U^{p^*}$, so $\Phi'(U)\ne0$.
\end{proof}

\begin{proof}[Proof of Theorem \ref{thm_index}]
\emph{Lower bound.} The function $U$ is positive, so it has exactly one nodal region, and \cite[Thm.~1.4]{FMW2019} gives $1\le i(U)$.

\emph{Upper bound.} By Lemma \ref{lem_nehari} we may fix $e\in X$ with $\Phi'(U)[e]=1$, and $T:=\ker \Phi'(U)$ is a closed subspace of $X$ of codimension one. Let $v\in T$. Since $\Phi$ is $C^2$ by Lemma \ref{lem_C2} and $\frac{\partial}{\partial s}\Phi(U+tv+se)\big|_{t=s=0}=1$, the implicit function theorem provides $\varepsilon>0$ and a function $s\in C^2((-\varepsilon,\varepsilon))$ with $s(0)=0$ and
\[
\Phi\big(\sigma(t)\big)=0\quad\text{for }|t|<\varepsilon,\qquad \sigma(t):=U+tv+s(t)e .
\]
Differentiating $\Phi(\sigma(t))=0$ at $t=0$ gives $\Phi'(U)[v]+s'(0)=0$, hence $s'(0)=0$ because $v\in T$. Shrinking $\varepsilon$ if necessary we also have $\sigma(t)\ne0$ for $|t|<\varepsilon$, since $\sigma$ is continuous and $U\ne0$. Thus $\sigma$ is a $C^2$ curve in $\mathcal N$ with $\sigma(0)=U$ and $\sigma'(0)=v$.

By Lemma \ref{lem_nehari} the function $g(t)=J(\sigma(t))$ has a minimum at $t=0$, so $g''(0)\ge0$. On the other hand $U$ is a critical point of $J$, so $J'(U)=0$ and
\[
g''(0)=J''(U)[\sigma'(0),\sigma'(0)]+J'(U)\big[\sigma''(0)\big]=J''(U)[v,v] .
\]
Therefore
\begin{equation}\label{eq_nonneg}
Q_U(v)=J''(U)[v,v]\ \ge\ 0\qquad\text{for every }v\in T .
\end{equation}

Let now $V\subset C_c^1(\mathbb{R}^N)$ be a subspace with $Q_U(\varphi)<0$ for every $\varphi\in V\setminus\{0\}$. Since $C_c^1(\mathbb{R}^N)\subset X$, \eqref{eq_nonneg} forces $V\cap T=\{0\}$. The restriction to $V$ of the quotient map $X\to X/T$ is therefore injective, and $\dim(X/T)=1$, so $\dim V\le1$. Hence $i(U)\le1$, and with the lower bound $i(U)=1$.
\end{proof}

\begin{remark}
For $p=2$ the statement and the proof are classical, the only point where $p>2$ is used being the twice differentiability of $J$ in Lemma \ref{lem_C2}. For $1<p<2$ the weight $|\nabla u|^{p-2}$ appearing in the second derivative is unbounded where $\nabla u$ vanishes, the bound \eqref{eq_Mbound} is of no use, and the argument above does not apply. It therefore does not determine $i(u_{\lambda,x_0})$ in that range; to the best of our knowledge, the exact value is not presently available.
\end{remark}

\begin{remark}
Combining Theorems \ref{thm_sob_endpoint} and \ref{thm_index} with \cite[Thm.~1.5]{FMW2019} gives, for $2<p<N$, a complete description of the low-index part of the Sobolev endpoint without assuming finite energy in advance: a $C^1$ solution of $-\Delta_pu=|u|^{p^*-2}u$ which is stable outside a compact set and satisfies $i(u)\le1$ is either $0$ or one of the functions \eqref{eq_bubble}, and in the latter case $i(u)=1$ exactly. Higher-index sign-changing solutions are not covered by such a formula classification; Theorem \ref{thm_sob_endpoint} reduces their study to the finite-energy critical problem and bounds their number of nodal regions by their Morse index.
\end{remark}

\begin{remark}[Morse index and nondegeneracy of the bubble]
Theorem \ref{thm_index} complements the nondegeneracy theorem of Pistoia and Vaira \cite{PistoiaVaira2021}. In the natural weighted space
\[
\mathcal D_*^{1,2}(\mathbb R^N)
=\overline{C_c^1(\mathbb R^N)}^{\,\left(\int |\nabla U|^{p-2}|\nabla\phi|^2\right)^{1/2}},
\]
they proved that the kernel of the linearized equation at a positive Aubin--Talenti extremal $U$ is generated exactly by the Jacobi fields coming from the symmetries,
\[
Z_0=\frac{N-p}{p}U+x\cdot\nabla U,
\qquad Z_i=\partial_{x_i}U,\quad i=1,\ldots,N.
\]
Thus, for $2<p<N$, the bubble has Morse index one in the sense of Definition \ref{def_stable}, while its linearized kernel in $\mathcal D_*^{1,2}(\mathbb R^N)$ has dimension $N+1$ and consists precisely of the scaling and translation modes.
\end{remark}

\section*{Declarations}

\smallskip
\noindent\textbf{Conflict of interest.} The author declares that he has no conflict of interest.

\smallskip
\noindent\textbf{Data availability.} Data sharing is not applicable to this article, as no datasets were generated or analyzed during the current study.

\smallskip
\noindent\textbf{Use of AI tools.} AI-assisted tools were used during the preparation of the manuscript. The author reviewed and edited all AI-assisted output and takes full responsibility for the content of the work.

\end{document}